\documentclass[11pt]{amsart}

\usepackage{amsmath,amssymb,enumitem,graphicx,url,amsaddr}

\usepackage[mathlines,modulo]{lineno}

\usepackage[colorlinks=false]{hyperref}

\usepackage{todonotes}

\usepackage{cleveref}

\theoremstyle{plain}
\newtheorem{theorem}{Theorem}[section]
\newtheorem{lemma}[theorem]{Lemma}
\newtheorem{proposition}[theorem]{Proposition}
\newtheorem{corollary}[theorem]{Corollary}

\newtheorem{conjecture}[theorem]{Conjecture}

\theoremstyle{definition}
\newtheorem{definition}[theorem]{Definition}
\newtheorem{claim}{Claim}[theorem]

\newtheorem{remark}[theorem]{Remark}

\newcommand{\mcal}[1]{\ensuremath{\mathcal{#1}}}
\newcommand{\mbb}[1]{\ensuremath{\mathbb{#1}}}
\newcommand{\closure}[1]{\ensuremath{\langle #1 \rangle}}
\newcommand{\formula}[1]{\text{\fontfamily{cmss}\selectfont \textup{#1}}}

\newcommand{\dash}{\nobreakdash-\hspace{0pt}}
\newcommand{\cl}{\operatorname{cl}}
\newcommand{\bw}{\operatorname{bw}}
\newcommand{\dw}{\operatorname{dw}}
\newcommand{\ind}{\formula{Ind}}
\newcommand{\sing}{\formula{Sing}}
\newcommand{\emp}{\formula{Empty}}
\newcommand{\vertex}{\operatorname{Vert}}

\newcommand{\enc}{\operatorname{enc}}
\newcommand{\desc}{\operatorname{desc}}
\newcommand{\image}{\operatorname{Im}}
\newcommand{\Rep}{\operatorname{Rep}}
\newcommand{\mso}{\ensuremath{\mathit{MS}_{0}}}
\newcommand{\cmso}{\ensuremath{\mathit{CMS}_{0}}}

\newcommand{\msone}{\ensuremath{\mathit{MS}_{1}}}
\newcommand{\mstwo}{\ensuremath{\mathit{MS}_{2}}}
\newcommand{\hgg}{$H$\dash gain-graphic}
\newcommand{\ba}{\backslash}
\newcommand{\dep}{\formula{dep}}
\newcommand{\indep}{\formula{indep}}

\title[Tree automata and matroids]
{Tree automata and pigeonhole classes of matroids: I}

\author[Funk]{Daryl Funk}
\address{Department of Mathematics,
Douglas College, Vancouver.}
\email{funkd@douglascollege.ca}

\author[Mayhew]{Dillon Mayhew$^{*}$}
\address{School of Mathematics and Statistics,
Victoria University of Wellington,
New Zealand.
Corresponding author: dillon.mayhew@vuw.ac.nz.}
\email[Corresponding author]{dillon.mayhew@vuw.ac.nz}

\author[Newman]{Mike Newman}
\address{Department of Mathematics and Statistics,
University of Ottawa.}
\email{mnewman@uottawa.ca}

\date{\today}

\begin{document}

\begin{abstract}
Hlin\v{e}n\'{y}'s Theorem shows that any sentence in the
monadic second-order logic of matroids can be tested in
polynomial time, when the input is limited to a class of
\mbb{F}\dash representable matroids with bounded branch-width
(where \mbb{F} is a finite field).
If each matroid in a class can be decomposed by a subcubic tree
in such a way that only a bounded amount of information
flows across displayed separations, then the class
has bounded decomposition-width.
We introduce the pigeonhole property for classes of matroids:
if every subclass with bounded branch-width also has
bounded decomposition-width, then the class is pigeonhole.
An efficiently pigeonhole class has a stronger property,
involving an efficiently-computable equivalence relation on subsets of the ground set.
We show that Hlin\v{e}n\'{y}'s Theorem extends to any
efficiently pigeonhole class.
In a sequel paper, we use these ideas to extend Hlin\v{e}n\'{y}'s Theorem
to the classes of fundamental transversal matroids, lattice path matroids,
bicircular matroids, and \hgg\ matroids, where $H$ is any finite group.
We also give a characterisation of the families of hypergraphs that can be
described via tree automata: a family is defined by a tree automaton
if and only if it has bounded decomposition-width.
Furthermore, we show that if a class of matroids has the pigeonhole
property, and can be defined in monadic second-order logic,
then any subclass with bounded branch-width has a
decidable monadic second-order theory.
\end{abstract}

\maketitle

\section{Introduction}

The \emph{model-checking} problem involves
a class of structures and a logical language
capable of expressing statements about those structures.
We consider a sentence from the language.
The goal is a procedure which will decide whether or not the sentence
is satisfied by a given structure from the class.
Our starting points are the model-checking meta-theorems due to Courcelle \cite{Cou90} and Hlin\v{e}n\'{y} \cite{Hli06c}.
Courcelle's Theorem proves that there is an efficient model-testing procedure
for any sentence in monadic second-order logic, when the
input class consists of graphs with bounded structural complexity.
Hlin\v{e}n\'{y} proves an analogue for matroids representable over finite fields.

Both theorems provide algorithms that are not only polynomial-time,
but \emph{fixed-parameter tractable} (see \cite{DF99}).
This means that the input contains a numerical parameter, $\lambda$.
The notion of fixed-parameter tractability captures the distinction between
running times of order $n^{f(\lambda)}$ and those of order $f(\lambda)n^{c}$,
where $n$ is the size of the input, $f(\lambda)$ is a value depending only on $\lambda$,
and $c$ is a constant.
When we restrict to a fixed value of $\lambda$, both running times are
polynomial with respect to $n$, but algorithms of the latter type will typically be feasible for a
larger range of $\lambda$\dash values.
An algorithm with a running time of $O(f(\lambda)n^{c})$ is said to be
fixed-parameter tractable with respect to $\lambda$.

\begin{theorem}[Courcelle's Theorem]
\label{bridge}
Let $\psi$ be a sentence in \mstwo.
We can test whether graphs satisfy $\psi$ with an algorithm that is fixed-parameter tractable with respect to tree-width.
\end{theorem}

The monadic second-order logic \mstwo\ allows us to
quantify over variables representing vertices, edges, sets of vertices, and
sets of edges.
$\text{NP}$\dash complete properties such as
Hamiltonicity and $3$\dash colourability can be
expressed in \mstwo.
Courcelle's Theorem shows that the extra structure imposed by bounding the
tree-width of input graphs transforms these properties from
being computationally intractable to tractable.

\begin{theorem}[Hlin\v{e}n\'{y}'s Theorem]
\label{record}
Let $\psi$ be a sentence in \cmso\ and let \mbb{F} be a finite field.
We can test whether \mbb{F}\dash representable matroids
satisfy $\psi$ with an algorithm that is fixed-parameter tractable with respect to branch-width.
\end{theorem}

The \emph{counting monadic second-order language} \cmso\ is described in \Cref{logic}.
Our main theorem identifies the structural properties underlying the proof of
Hlin\v{e}n\'{y}'s Theorem.

\begin{theorem}
\label{lounge}
Let \mcal{M} be an efficiently pigeonhole class of matroids.
Let $\psi$ be a sentence in \cmso.
We can test whether matroids in \mcal{M} satisfy $\psi$ with an algorithm that is fixed-parameter tractable with respect to branch-width.
\end{theorem}

\Cref{lounge} is proved by \Cref{ermine} and \Cref{jibber}.
In a sequel \cite{FMN-II}, we will prove that we can now
extend Hlin\v{e}n\'{y}'s Theorem to several natural classes of matroids.
Fundamental transversal matroids,
lattice path matroids, bicircular matroids, 
and \hgg\ matroids with $H$ a finite group:
all these classes have fixed-parameter tractable algorithms for \cmso\ model-checking, where the parameter is branch-width.

The pigeonhole property is motivated by matroids representable over finite fields.
Let $(U,V)$ be a separation of order at most $\lambda$ in $M$,
a simple matroid representable over a finite field, \mbb{F}.
We can think of $M$ as a subset of points in the projective
space $P=\operatorname{PG}(r(M)-1,\mbb{F})$.
The subspaces of $P$ spanned by $U$ and $V$ intersect in a
subspace, $P'$, with affine dimension at most $\lambda-2$.
If $X$ and $X'$ are subsets of $U$, and their spans intersect
$P'$ in the same subspace, then no subset of $V$
can distinguish them.
By this we mean that both $X\cup Z$ and $X'\cup Z$ are independent
or both are dependent, for any subset $Z\subseteq V$.
This induces an equivalence relation on subsets of $U$.
The number of classes under this relation is at most the
number of subspaces of $\operatorname{PG}(\lambda-2,\mbb{F})$.

Now we generalise this idea.
Let $E$ be a finite set, and let \mcal{I} be a collection of subsets.
If $U$ is a subset of $E$, then $\sim_{U}$ is the equivalence relation
on subsets of $U$ such that $X\sim_{U} X'$ if no subset of $E-U$
can distinguish between $X$ and $X'$; that is,
for all $Z\subseteq E-U$, both $X\cup Z$ and $X'\cup Z$
are in \mcal{I}, or neither of them is.
A set-system, $(E,\mcal{I})$, has \emph{decomposition-width} at most $q$ if
there is a subcubic tree with leaves in bijection with $E$,
such that if $U$ is any set displayed by the tree, then $\sim_{U}$ has
at most $q$ equivalence classes.
Since every matroid is a set-system, the decomposition-width
of a matroid is a natural specialisation.
This notion of decomposition-width is equivalent to that used by
Kr\'{a}l \cite{Kra12} and by Strozecki \cite{Str10,Str11},
although our definition is cosmetically quite different.

\Cref{lounge} relies on tree automata to check whether
monadic sentences are satisfied.
(As do the theorems of Courcelle and Hlin\v{e}n\'{y}.)
Tree automata also provide further evidence that the
notion of decomposition-width is a natural one, as
we see in the next \namecref{mtheorem1}.
In our conception, a tree automaton processes
a tree from leaves to root, applying a state to each node.
Each node of the tree is initially labelled with a character from a
finite alphabet, and the state applied to a node depends
on the character written on that node, as well as the
states that have been applied to its children.
The automaton accepts or rejects the tree according to the state
it applies to the root.
The characters applied to the leaves 
can encode a subset of the leaves, so we can
think of the automaton as either accepting or rejecting each
subset of the leaves.
Thus each tree automaton gives rise to a family of set-systems.
The ground set of such a set-system is the set of leaves of
a tree, and a subset belongs to the system if it
is accepted by the automaton.
We say that a family of set-systems is
\emph{automatic} if there is an automaton which
produces the family in this way (\Cref{dancer}).
It is natural to ask which families of set-systems are
automatic, and we answer this question in \Cref{mainproof}.

\begin{theorem}
\label{mtheorem1}
A class of set-systems is automatic if and only if it has
bounded decomposition-width.
\end{theorem}

A class of matroids with bounded decomposition-width must
have bounded branch-width (\Cref{whaler}).
The converse does not hold (\cite[Lemma 4.1]{FMN-II}).
If \mcal{M} is a class of matroids and every subclass with
bounded branch-width also has bounded decomposition-width,
then \mcal{M} is \emph{pigeonhole} (\Cref{juicer}).
The class of lattice path matroids has the pigeonhole
property (\cite[Theorem 7.2]{FMN-II}).
Other natural classes have an even stronger property.
Let \mcal{M} be a class of matroids.
Assume there is a value, $\pi(\lambda)$, for every positive integer
$\lambda$, such that the following holds:
We let $M$ be a matroid in \mcal{M}, and we let $U$ be a subset of $E(M)$.
If $\lambda_{M}(U)$, the connectivity of $U$, is at most $\lambda$, then
$\sim_{U}$ has at most $\pi(\lambda)$ equivalence classes.
Under these circumstances, we say \mcal{M} is
\emph{strongly pigeonhole} (\Cref{fizzer}).
The matroids representable over a finite field
(\cite[Theorem 5.1]{FMN-II}) and fundamental transversal matroids
(\cite[Theorem 6.3]{FMN-II}) are strongly pigeonhole classes.
An \emph{efficiently pigeonhole} class is strongly
pigeonhole, and has the additional property that
we can efficiently compute a relation that refines $\sim_{U}$
(\Cref{yakuza}).

Now we describe the structure of this article.
\Cref{pigeonhole} discusses decomposition-width, 
pigeonhole classes, and strongly pigeonhole classes.
In \Cref{logic} we describe the monadic logic \cmso.
\Cref{automatic} develops the necessary tree automaton ideas.
\Cref{mainproof} is dedicated to the proof of \Cref{mtheorem1}.
In \Cref{complexity} we prove \Cref{lounge}.
We also show that \Cref{lounge} holds under
the weaker condition that the $3$\dash connected matroids
in \mcal{M} form an efficiently pigeonhole class
(\Cref{marker}).
(However, we require that we can efficiently compute a
description of any minor of the input matroid, so this is not
a true strengthening of \Cref{lounge}.)
\Cref{marker} is necessary because we do not know that bicircular
matroids or \hgg\ matroids (with $H$ finite) form
efficiently pigeonhole classes.
(Although we conjecture this is the case \cite[Conjecture 9.3]{FMN-II}.)
However, we do know that the subclasses consisting of $3$\dash connected matroids are efficiently pigeonhole (\cite[Theorem 8.4]{FMN-II}).

In the final section (\Cref{decidability}), we consider the
question of \emph{decidability}.
A class of set-systems has a decidable monadic second-order
theory if there is a Turing Machine
(not time-constrained) which will take any sentence
as input, and decide whether it is satisfied by all systems in the
class.
The main result of this section says that if a
class of matroids has the pigeonhole
property and can be defined by a sentence in
monadic second-order logic, then any subclass with
bounded branch-width has a decidable theory (\Cref{quaker}).
The special case of \mbb{F}\dash representable matroids
(\mbb{F} finite) has been noted by Hlin\v{e}n\'{y} and
Seese \cite[Corollary 5.3]{HS06}.
On the other hand, the class of
\mbb{K}\dash representable rank\dash $3$ matroids has
an undecidable theory when \mbb{K} is an infinite field
(\Cref{pulpit}).

A good introduction to automata can be found in \cite{Eng15}.
For the basic concepts and notation of matroid theory, we rely on
\cite{Oxl11}.
Recall that if $M$ is a matroid, and
$(U,V)$ is a partition of $E(M)$, then
$\lambda_{M}(U)$ is $r_{M}(U)+r_{M}(V)-r(M)$.
Note that $\lambda_{M}(U)=\lambda_{M}(V)$ and
$\lambda_{M}(U)\leq r(M)$.
A set, $U$, is \emph{$k$\dash separating} if $\lambda_{M}(U)<k$,
and a \emph{$k$\dash separation} is a partition, $(U,V)$, of the
ground set such that $|U|,|V|\geq k$, and both $U$ and $V$ are
$k$\dash separating.

\section{Pigeonhole classes}
\label{pigeonhole}

Now we introduce one of our principal definitions.
A class of set-systems has bounded decomposition-width if those
set-systems can be decomposed by subcubic trees in such a way that
only a bounded amount of information flows
across any of the displayed separations.
This section is dedicated to formalising these ideas.

\begin{definition}
\label{hybrid}
A \emph{set-system} is a pair $(E,\mcal{I})$ where $E$ is a finite set
and \mcal{I} is a family of subsets of $E$.
We refer to $E$ as the \emph{ground set}, and the members of \mcal{I}
as \emph{independent} sets.
\end{definition}

In some circumstances, a set-system might be called a hypergraph and the independent sets
might be called hyperedges.
We prefer more matroid-oriented language.

\begin{definition}
\label{fracas}
Let $(E,\mcal{I})$ be a set-system, and let $U$ be a subset of $E$.
Let $X$ and $X'$ be subsets of $U$.
We say $X$ and $X'$ are \emph{equivalent} (relative to $U$),
written $X\sim_{U} X'$, if for every subset $Z\subseteq E-U$,
the set $X\cup Z$ is in \mcal{I} if and only if $X'\cup Z$ is in \mcal{I}.
\end{definition}

Informally, we think of $X\sim_{U} X'$ as meaning that no subset of $E-U$ can
`distinguish' between $X$ and $X'$.
It is clear that $\sim_{U}$ is an equivalence relation on subsets of $U$.
Note that by taking $Z$ to be the empty set, we can see that no member
of \mcal{I} is equivalent to a subset not in \mcal{I}.
Assuming that \mcal{I} is closed under subset containment
(as would be the case if \mcal{I} were the family of
independent sets in a matroid), then all subsets of $U$ that are not in \mcal{I}
are equivalent.

\begin{proposition}
\label{icecap}
Let $(E,\mcal{I})$ be a set-system and let $U$ and $V$ be disjoint subsets of $E$.
If $X\sim_{U} X'$ and $Y\sim_{V} Y'$, then
$(X\cup Y)\sim_{(U\cup V)} (X'\cup Y')$.
In particular, $X\cup Y$ belongs to \mcal{I} if and only
if $X'\cup Y'$ does.
\end{proposition}

\begin{proof}
Let $Z$ be an arbitrary subset of $E-(U\cup V)$, and assume that
$X\cup Y\cup Z$ is in \mcal{I}.
Because $Y\cup Z\subseteq E-U$ and $X\sim_{U}X'$,
it follows that $X'\cup Y\cup Z$ is in \mcal{I}.
Now $X'\cup Z\subseteq E-V$ and $Y\sim_{V}Y'$, so
$X'\cup Y'\cup Z$ is in \mcal{I}.
By an identical argument, we see that
if $X'\cup Y'\cup Z$ is in \mcal{I}, then so is $X\cup Y\cup Z$.
\end{proof}

In the previous result, if $(U,V)$ is a partition of $E$, then $X\cup Y$ will be equivalent
to $X'\cup Y'$ under $\sim_{U\cup V}$ if and only if both $X\cup Y$ and $X'\cup Y'$ are
in \mcal{I}, or neither is.

\begin{proposition}
\label{shadow}
Let $(E,\mcal{I})$ be a set-system, and let $(U,V)$ be a partition of $E$.
If $q$ is the number of equivalence classes under $\sim_{U}$,
then the number of equivalence classes under $\sim_{V}$ is at
most $2^{q}$.
\end{proposition}

\begin{proof}
Let the equivalence classes under $\sim_{U}$ be
$\mcal{E}_{1},\ldots, \mcal{E}_{q}$,
and let $X_{i}$ be a member of $\mcal{E}_{i}$ for each $i$.
Let $Z$ be any subset of $V$.
We define $b(Z)$ to be the binary string of length
$q$, where the $i$th character is $1$ if and only if
$X_{i}\cup Z$ is in \mcal{I}.
It is clear that this string is well-defined and does not depend on our choice of
the representatives $X_{i}$.
We complete the proof by showing that when $Z,Z'\subseteq V$
satisfy $b(Z)=b(Z')$, they also satisfy $Z\sim_{V} Z'$.
Assume this is not the case, and let $X\subseteq U$ be
such that exactly one of $X\cup Z$ and $X\cup Z'$ is in \mcal{I}.
Without loss of generality, we assume $X\cup Z\in \mcal{I}$ and
$X'\cup Z\notin\mathcal{I}$.
Assume that $X$ is a member of $\mcal{E}_{i}$.
Since $b(Z)=b(Z')$, either both of $X_{i}\cup Z$ and $X_{i}\cup Z'$ are in \mcal{I}, or neither is.
In the first case, $X_{i}\cup Z'\in\mcal{I}$ and $X\cup Z'\notin\mcal{I}$, so we
contradict $X_{i}\sim_{U} X$.
In the second case, $X_{i}\cup Z\notin\mcal{I}$ and $X\cup Z\in\mcal{I}$,
so we reach the same contradiction.
\end{proof}

A \emph{subcubic} tree is one in which every vertex has degree three or one.
A degree-one vertex is a \emph{leaf}.
Let $M=(E,\mcal{I})$ be a set-system.
A \emph{decomposition} of $M$ is a pair $(T,\varphi)$, where $T$ is a
subcubic tree, and $\varphi$ is a bijection from $E$ into the set of leaves of $T$.
Let $e$ be an edge joining vertices $u$ and $v$ in $T$.
Then $e$ partitions $E$ into sets $(U_{e}, V_{e})$ in the following
way: an element $x\in E$ belongs to $U_{e}$ if and only if
the path in $T$ from $\varphi(x)$ to $u$ does not contain $v$.
We say that the partition $(U_{e},V_{e})$ and the sets $U_{e}$ and
$V_{e}$ are \emph{displayed} by the edge $e$.
Define $\dw(M;T,\varphi)$ to be the maximum
number of equivalence classes in $\sim_{U}$, where the maximum is taken
over all subsets, $U$, displayed by an edge in $T$.
Define $\dw(M)$ to be the minimum value of
$\dw(M;T,\varphi)$, where the minimum is taken over all
decompositions $(T,\varphi)$ of $M$.
This minimum is the \emph{decomposition-width} of $M$.
The notion of decomposition-width specialises to matroids in the
obvious way.

\begin{definition}
\label{zydeco}
Let $M$ be a matroid.
Then $\dw(M)$ is equal to $\dw(E(M),\mcal{I}(M))$.
\end{definition}

It is an exercise to show that a class of matroids has bounded decomposition-width if and only if it has bounded decomposition-width, as defined by Kr\'{a}l \cite{Kra12} and
Strozecki \cite{Str10,Str11}.
Kr\'{a}l states the next result without proof.

\begin{proposition}
\label{tubing}
Let $x$ be an element of the matroid $M$.
Then $\dw(M\ba x)\leq \dw(M)$ and $\dw(M/x)\leq \dw(M)$.
\end{proposition}

\begin{proof}
Let $(T,\varphi)$ be a decomposition of $M$ and assume that whenever
$U$ is a displayed set, then $\sim_{U}$ has no more than $\dw(M)$
equivalence classes.
Let $T'$ be the tree obtained from $T$ by deleting $\varphi(x)$ and then
contracting an edge so that every vertex in $T'$ has degree one or three.
Let $U$ be any subset of $E(M)-x$ displayed by $T'$.
Then either $U$ or $U\cup x$ is displayed by $T$.
Let $M'$ be either $M\backslash x$ or $M/x$.
We will show that in $M'$, the number of equivalence classes under
$\sim_{U}$ is no greater than the number of classes under $\sim_{U}$
or $\sim_{U\cup x}$ in $M$.
Let $X$ and $X'$ be representatives of distinct classes under $\sim_{U}$ in $M'$.
We will be done if we can show that these representatives correspond to distinct
classes in $M$.
Without loss of generality, we can assume that $Z$ is a subset of $E(M)-(U\cup x)$
such that $X\cup Z$ is independent in $M'$, but $X'\cup Z$ is dependent.
If $M'=M\backslash x$, then $X\cup Z$ is independent in $M$ and $X'\cup Z$ is
dependent, and thus we are done.
So we assume that $M'=M/x$.
If $U$ is displayed by $T$, then we observe that $X\cup (Z\cup x)$ is independent
in $M$, while $X'\cup (Z\cup x)$ is dependent.
On the other hand, if $U\cup x$ is displayed, then $(X\cup x)\cup Z$ is independent
in $M$ and $(X'\cup x)\cup Z$ is dependent.
Thus $X$ and $X'$ belong to distinct equivalence classes in $M$, as claimed.
\end{proof}

\Cref{tubing} shows that the class of matroids with decomposition-width at most $k$ is
minor-closed.

Let $M$ be a matroid.
The \emph{branch-width} of $M$ (written $\bw(M)$) is defined
as follows.
If $(T,\varphi)$ is a decomposition of $M=(E(M),\mcal{I}(M))$, then
$\bw(M;T,\varphi)$ is the maximum value of
\[\lambda_{M}(U_{e})+1=r_{M}(U_{e})+r_{M}(V_{e})-r(M)+1,\]
where the maximum is taken over all partitions $(U_{e},V_{e})$
displayed by edges of $T$.
Now $\bw(M)$ is the minimum value of $\bw(M;T,\varphi)$,
where the minimum is taken over all decompositions of $M$.
We next show that for classes of matroids, bounded
decomposition-width implies bounded branch-width.

\begin{proposition}
\label{healer}
Let $M$ be a matroid, and let $U$ be a subset of $E(M)$.
There are at least $\lambda_{M}(U)+1$ equivalence classes under
the relation $\sim_{U}$.
\end{proposition}

\begin{proof}
Define $V$ to be $E(M)-U$.
Let $\lambda$ stand for $\lambda_{M}(U)$, so that $\lambda=r(U)+r(V)-r(M)$.
We will prove that $\sim_{U}$ has at least $\lambda+1$ equivalence classes.
Let $B_{V}$ be a basis of $M|V$, and let $B$ be a basis of $M$
that contains $B_{V}$.
Then $B\cap U$ is independent in $M|U$, and
\begin{linenomath*}
\begin{multline*}
r(U)-|B\cap U|
=r(U)-(|B|-|B_{V}|)
=r(U)-(r(M)-r(V))\\
=r(U)-(r(U)-\lambda)
=\lambda.
\end{multline*}
\end{linenomath*}
Therefore we let $(B\cap U)\cup\{x_{1},\ldots, x_{\lambda}\}$ be a
basis of $M|U$, where $x_{1},\ldots, x_{\lambda}$ are distinct
elements of $U-B$.
Next we construct a sequence of distinct elements,
$y_{1},\ldots, y_{\lambda}$ from $B_{V}$ such that
$(B-\{y_{1},\ldots, y_{i}\})\cup\{x_{1},\ldots, x_{i}\}$ is a
basis of $M$ for each $i\in\{0,\ldots, \lambda\}$.
We do this recursively.
Let $C$ be the unique circuit contained in
\[
(B-\{y_{1},\ldots, y_{i}\})\cup\{x_{1},\ldots, x_{i}\}\cup x_{i+1}
\]
and note that $x_{i+1}$ is in $C$.
If $C$ contains no elements of $B_{V}$, then it is contained
in $(B\cap U)\cup\{x_{1},\ldots, x_{\lambda}\}$, which is impossible.
So we simply let $y_{i+1}$ be an arbitrary element in
$C\cap B_{V}$.

We complete the proof by showing that
\[(B\cap U)\cup\{x_{1},\ldots, x_{i}\}\quad\text{and}\quad
(B\cap U)\cup\{x_{1},\ldots, x_{j}\}\]
are inequivalent under $\sim_{U}$ whenever
$0\leq i<j\leq \lambda$.
Indeed, if $Z=B_{V}-\{y_{1},\ldots, y_{i}\}$, then
$(B\cap U)\cup\{x_{1},\ldots, x_{i}\}\cup Z$ is a basis of
$M$, and is properly contained in
$(B\cap U)\cup\{x_{1},\ldots, x_{j}\}\cup Z$, so the
last set is dependent, and we are done.
\end{proof}

\begin{corollary}
\label{whaler}
Let $M$ be a matroid.
Then $\dw(M)\geq \bw(M)$.
\end{corollary}

\begin{proof}
Assume that $\bw(M)>\dw(M)$.
Let $(T,\varphi)$ be a decomposition of $M$ such that
if $U$ is any set displayed by an edge of $T$, then
$\sim_{U}$ has at most $\dw(M)$ equivalence classes.
There is some edge $e$ of $T$ displaying a set
$U_{e}$ such that $\lambda_{M}(U_{e})+1>\dw(M)$,
for otherwise this decomposition of $M$ certifies that
$\bw(M)\leq \dw(M)$.
But $\sim_{U_{e}}$ has at least $\lambda_{M}(U_{e})+1$
equivalence classes by \Cref{healer}.
As $\lambda_{M}(U_{e})+1>\dw(M)$, this
contradicts our choice of $(T,\varphi)$.
\end{proof}

It is easy to see that the class of rank\dash $3$ sparse paving matroids
has unbounded decomposition width (see \cite[Lemma 4.1]{FMN-II}),
so the converse of \Cref{whaler} does not hold.
Kr\'{a}l proved the special case of \Cref{whaler} when $M$
is representable over a finite
field \cite[Theorem 2]{Kra12}.

Since we would like to consider natural classes of matroids
that have unbounded branch-width, we are motivated to make the next definition.

\begin{definition}
\label{juicer}
Let \mcal{M} be a class of matroids.
Then \mcal{M} is \emph{pigeonhole} if, for every positive
integer, $\lambda$, there is an integer $\rho(\lambda)$
such that $\bw(M)\leq \lambda$ implies
$\dw(M)\leq \rho(\lambda)$, for every $M\in\mcal{M}$.
\end{definition}

Thus a class of matroids is pigeonhole if every subclass with bounded
branch-width also has bounded decomposition-width.
The class of \mbb{F}\dash representable matroids is pigeonhole
when \mbb{F} is a finite field \cite[Theorem 5.1]{FMN-II}.
Note that the class of \mbb{F}\dash representable
matroids certainly has unbounded decomposition-width, since
it has unbounded branch-width.
Some natural classes possess a stronger property than the
pigeonhole property:

\begin{definition}
\label{fizzer}
Let \mcal{M} be a class of matroids.
Assume that for every positive integer $\lambda$, there is
a positive integer $\pi(\lambda)$, such that whenever $M\in\mcal{M}$
and $U\subseteq E(M)$ satisfies $\lambda_{M}(U)\leq \lambda$,
there are at most $\pi(\lambda)$ equivalence classes under $\sim_{U}$.
In this case we say that \mcal{M} is \emph{strongly pigeonhole}.
\end{definition}

\begin{proposition}
\label{fedora}
If a class of matroids is strongly pigeonhole, then it is pigeonhole.
\end{proposition}

\begin{proof}
Let \mcal{M} be a strongly pigeonhole class, and let $\pi$ be the function
from \Cref{fizzer}.
We may as well assume that $\pi$ is non-decreasing.
Let $\lambda$ be any positive integer, and
let $M$ be a matroid in \mcal{M} with branch-width at most $\lambda$.
Let $(T,\varphi)$ be a decomposition of $M$ such that
$\lambda_{M}(U)+1\leq \lambda$ for any set $U$ displayed by an
edge of $T$.
Then there are at most $\pi(\lambda-1)$\dash equivalence classes under
$\sim_{U}$.
Thus $(T,\varphi)$ demonstrates that $\dw(M)\leq \pi(\lambda-1)$.
So $\bw(M)\leq \lambda$ implies $\dw(M)\leq \pi(\lambda-1)$
for each $M\in \mcal{M}$, and the result follows.
\end{proof}

\begin{remark}
\label{effect}
To see that the strong pigeonhole property is strictly stronger
than the pigeonhole property, let \mcal{M} be the class of rank\dash two matroids.
Let $M$ be a member of \mcal{M} with $t$ parallel pairs (where $t\geq 2$).
Let $U$ be a set that contains exactly one element from each of these
pairs.
Then $\lambda_{M}(U)=2$.
However, it is easy to demonstrate that there are at least $t$ equivalence classes
under $\sim_{U}$, so this number is unbounded.
This demonstrates that \mcal{M} is not strongly pigeonhole.
However, if $M$ is in \mcal{M}, then there is a decomposition of $M$ such that
whenever $(U,V)$ is a displayed partition, at most one parallel class contains
elements of both $U$ and $V$.
Now we easily check that $\sim_{U}$ has at most five equivalence classes,
so $\dw(M)\leq 5$ for all $M\in \mcal{M}$, implying that \mcal{M} is
pigeonhole.
\end{remark}

\section{Monadic logic}
\label{logic}

In this section we construct the formal language \cmso\ (\emph{counting monadic second-order logic}).
We give ourselves a countably infinite supply of variables:
$X_{1},X_{2},X_{3},\ldots$\,.
We have a unary predicate: $\ind$, and
one binary predicate: $\subseteq$.
Furthermore, for each pair of integers $p$ and $q$ satisfying $0\leq p < q$, we have the unary predicate $|\cdot|_{p,q}$.
We use the standard connectives $\land$ and $\neg$, and
the quantifier $\exists$.
The \emph{atomic formulas} have the form
$\ind(X_{i})$, $X_{i}\subseteq X_{j}$, or $|X_{i}|_{p,q}$.
The atomic formulas $\ind(X_{i})$ and $|X_{i}|_{p,q}$ have $X_{i}$ as their \emph{free variable},
whereas the free variables of $X_{i}\subseteq X_{j}$ are $X_{i}$ and $X_{j}$.
A \emph{formula} is constructed by a finite application of the
following rules:
\begin{enumerate}[label=\textup{(\roman*)}]
\item an atomic formula is a formula,
\item if $\psi$ is a formula, then $\neg\psi$ is a formula
with the same free variables as $\psi$,
\item if $\psi$ is a formula, and $X_{i}$ is a free variable
in $\psi$, then $\exists X_{i} \psi$ is a formula;
its free variables are the free variables of $\psi$ except for
$X_{i}$, which is a \emph{bound variable} of
$\exists X_{i} \psi$,
\item if $\psi$ and $\phi$ are formulas, and no variable
is free in one of $\psi$ and $\phi$ and bound in the other,
then $\psi\land\phi$ is a formula, and its free variables are exactly
those that are free in either $\psi$ or $\phi$.
(We can rename bound variables, so this restriction does
not significantly constrain us.)
\end{enumerate}

Then \cmso\ is the collection of all formulas.
A formula is a \emph{sentence} if it has no free variables,
and is \emph{quantifier-free} if it has no bound variables.

\begin{definition}
\label{patent}
\emph{Monadic second-order logic}, denoted by \mso, is the collection of formulas that can be constructed without using any predicate of the form $|\cdot|_{p,q}$.
\end{definition}

Let $(E,\mcal{I})$ be a set-system.
Let $\psi$ be a formula in \cmso\ and
let $F$ be the set of free variables in $\psi$.
An \emph{interpretation} of $\psi$ in $(E,\mcal{I})$ is a
function $\theta$ from $F$ into the power set of $E$.
We think of $\theta$ as a set of ordered pairs with the first element
being a variable in $F$ and the second being a subset of $E$.
We define what it means for the pair $(E,\mcal{I})$ to \emph{satisfy}
$\psi$ under the interpretation $\theta$.
If $\psi$ is $\ind(X_{i})$, then $(E,\mcal{I})$ satisfies $\psi$
if $\theta(X_{i})$ is in \mcal{I}.
If $\psi$ is $|X_{i}|_{p,q}$, then $(E,\mcal{I})$ satisfies $\psi$ if $|\theta(X_{i})|$ is equivalent to $p$ modulo $q$.
Similarly, $X_{i}\subseteq X_{j}$ is satisfied if
$\theta(X_{i})\subseteq \theta(X_{j})$.
Now we extend this definition to formulas that are not
atomic.

If $\psi=\neg\phi$, then $(E,\mcal{I})$ satisfies $\psi$
if and only if it does not satisfy $\phi$ under $\theta$.
If $\psi=\phi_{1}\land\phi_{2}$, then $\psi$ is satisfied
if $(E,\mcal{I})$ satisfies both $\phi_{1}$ and $\phi_{2}$
under the interpretations consisting of $\theta$ restricted
to the free variables of $\phi_{1}$ and $\phi_{2}$.
Finally, if $\psi=\exists X_{i}\phi$, then $(E,\mcal{I})$ satisfies $\psi$
if and only if there is a subset $Y_{i}\subseteq E$ such that
$(E,\mcal{I})$ satisfies $\phi$ under the interpretation
$\theta\cup\{(X_{i},Y_{i})\}$.

We use $\psi\lor\phi$ as shorthand for $\neg((\neg\psi)\land(\neg\phi))$,
and $\psi\to\phi$ as shorthand for $(\neg\psi)\lor \phi$.
The formula $\psi \leftrightarrow \phi$ is shorthand for
$(\psi\to \phi)\land (\phi\to\psi)$.
If $X_{i}$ is a free variable in $\psi$, then $\forall X_{i} \psi$ stands for
$\neg \exists X_{i} \neg\psi$. 
The predicate $\emp(X_{i})$ stands for
\[
\forall X (X\subseteq X_{i} \to X_{i}\subseteq X)
\]
and is satisfied exactly when $X_{i}$ is interpreted as the empty set.
(Here $X$ is a variable not equal to $X_{i}$.)
Similarly, $\sing(X_{i})$ stands for
\[
\neg \emp(X_{i}) \land \forall X (X\subseteq X_{i} \to (\emp(X) \lor X_{i}\subseteq X))
\]
and is satisfied exactly when $X_{i}$ is interpreted as a singleton set.

As is demonstrated in \cite{MNW18}, there are \mso\ sentences
that are satisfied by $(E,\mcal{I})$ if and only if
\mcal{I} is the family of independent sets of a matroid.
Furthermore, there are \mso\ sentences that characterise
any minor-closed class of matroids having only finitely
many excluded minors
(see \cite{MNW18} or \cite[Lemma 5.1]{Hli03b}).
On the other hand, the main theorem of \cite{MNW18} shows that no
\mso\ sentence characterises the class of representable matroids,
or the class of \mbb{K}\dash representable matroids when \mbb{K}
is an infinite field.

\section{Automatic classes}
\label{automatic}

Our second principal definition involves families of set-systems 
that can be encoded by a tree, where that tree can be processed
by a machine that simulates an independence oracle.
We start by introducing tree automata.
We use \cite{Eng15} as a general reference.

\begin{definition}
\label{sitcom}
Let $T$ be a tree with a distinguished \emph{root} vertex, $t$.
Assume that every vertex of $T$ other than $t$ has degree one or three,
and that if $T$ has more than one vertex, then $t$ has degree two.
The \emph{leaves} of $T$ are the degree-one vertices.
In the case that $t$ is the only vertex, we also consider $t$ to be a leaf.
Let $L(T)$ be the set of leaves of $T$.
If $T$ has more than one vertex, and $v$ is a non-leaf, then $v$ is adjacent
with two vertices that are not in the path from $v$ to $t$.
These two vertices are the \emph{children} of $v$.
We distinguish the \emph{left} child and the \emph{right} child of $v$.
Now let $\Sigma$ be a finite \emph{alphabet} of \emph{characters}.
Let $\sigma$ be a function from $V(T)$ to $\Sigma$.
Under these circumstances we say that $(T,\sigma)$ is a \emph{$\Sigma$\dash tree}.
\end{definition}

\begin{definition}
\label{nuance}
A \emph{tree automaton} is a tuple  $(\Sigma, Q, F, \delta_{0},\delta_{2})$,
where $\Sigma$ is a finite alphabet, and $Q$ is a finite set of \emph{states}.
The set of \emph{accepting states} is a subset $F\subseteq Q$.
We say $\delta_{0}$ and $\delta_{2}$ are \emph{transition rules}:
$\delta_{0}$ is a partial function from $\Sigma$ to $2^{Q}$ and
$\delta_{2}$ is a partial function from $\Sigma \times Q\times Q$ to $2^{Q}$.
\end{definition}

We think of the automaton as processing the vertices in a $\Sigma$\dash tree,
from leaves to root, applying a set of states to each vertex.
The set of states applied to a leaf, $v$, is given by the image of $\delta_{0}$,
applied to the $\Sigma$\dash label of $v$.
For a non-leaf vertex, $v$, we apply $\delta_{2}$ to the tuple consisting of
the $\Sigma$\dash label of $v$, a state applied to the left child, and a state applied
to right child.
We take the union of all such outputs, as we range over all states applied to
the children of $v$, and this union is the set we apply to $v$.

More formally, let $A=(\Sigma, Q, F, \delta_{0},\delta_{2})$ be an automaton.
Let $(T,\sigma)$ be a $\Sigma$\dash tree with root $t$.
We let $r\colon V(T)\to 2^{Q}$ be the function recursively defined
as follows:
\begin{enumerate}[label=\textup{(\roman*)}]
\item if $v$ is a leaf of $T$, then $r(v)$ is $\delta_{0}(\sigma(v))$
if this is defined, and is otherwise the empty set.
\item if $v$ has left child $v_{L}$ and right child $v_{R}$, then
\[
r(v)=\bigcup_{(q_{L},q_{R})\in r(v_{L})\times r(v_{R})} \delta_{2}(\sigma(v),q_{L},q_{R}),
\]
as long as the images in this union are all defined: if they are
not then we set $r(v)$ to be the empty set.
\end{enumerate}

We say that $r$ is the \emph{run} of the automaton $A$ on $(T,\sigma)$.
Note that we define a union taken over an empty collection
to be the empty set.
Thus if a child of $v$ has been assigned an empty set of states, then
$v$ too will be assigned an empty set of states.
We say that $A$ \emph{accepts} $(T,\sigma)$ if $r(t)$ contains
an accepting state.

The automaton, $A=(\Sigma, Q, F, \delta_{0},\delta_{2})$, is \emph{deterministic} if every set in the images of $\delta_{0}$ and $\delta_{2}$ is a singleton.
The next result shows that non-determinism in fact gives us no extra computing power.
The idea here dates to Rabin and Scott \cite{RS59}
(see \cite[Theorem 12.3.1]{DF13}).

\begin{lemma}
\label{folder}
Let $A'=(\Sigma, Q, F', \delta_{0}',\delta_{2}')$ be a tree automaton.
There exists a deterministic tree automaton,
$A=(\Sigma, 2^{Q},F,\delta_{0},\delta_{2})$,
such that $A'$ and $A$ accept exactly the same $\Sigma$\dash trees.
\end{lemma}

\begin{proof}
Note that the states in $A$ are sets of states in $A'$.
Let $F$ be $\{X\in 2^{Q}\colon X\cap F'\ne \emptyset\}$.
Thus a state is accepting in $A$ if and only if it contains an
accepting state of $A'$.
For each $\sigma\in \Sigma$, we define $\delta_{0}(\sigma)$ to be
$\{\delta_{0}'(\sigma)\}$ when $\delta_{0}'(\sigma)$ is defined.
For any $\sigma\in \Sigma$, and any $X,Y\in 2^{Q}$, we set
\[
\delta_{2}(\sigma,X,Y)=\left\{ \bigcup_{(q_{L},q_{R})\in X\times Y}
\delta_{2}'(\sigma,q_{L},q_{R}) \right\}
\]
as long as every image in the union is defined.
Thus every image of $\delta_{0}$ or $\delta_{2}$ is a singleton
set, so $A$ is deterministic, as desired.

Let $(T,\sigma)$ be a $\Sigma$\dash tree with root $t$.
Let $r'$ and $r$ be the runs of $A'$ and $A$ on $(T,\sigma)$.
We easily establish that $r(v)=\{r'(v)\}$, for each vertex $v$.
If $A'$ accepts $(T,\sigma)$, then $r'(t)$ contains a state in $F'$.
Therefore $r'(t)$ is a member of $F$, so $r(t)=\{r'(t)\}$ contains a member of $F$.
Hence $A$ also accepts $(T,\sigma)$.
For the converse, assume that $A$ accepts $(T,\sigma)$.
Then $r(t)=\{r'(t)\}$ contains an accepting state.
This means that $r'(t)$ is not disjoint from $F'$, so
$A'$ also accepts $(T,\sigma)$, and we are done.
\end{proof}

We would like to use tree automata to decide if a
formula in \cmso\ is satisfied by a set-system, $(E,\mcal{I})$.
This formula may have free variables, and in this case
deciding whether the formula is satisfied only makes
sense if we assign subsets of $E$ to the free variables.
So our next job is to formalise a way to encode this
assignment into the leaf labels of a tree.

Let $I$ be a finite set of positive integers.
We use $\{0,1\}^{I}$ to denote the set of functions from $I$ into $\{0,1\}$.
If $I$ is empty, then $\{0,1\}^{I}$ is the empty set.
Let $\Sigma$ be a finite alphabet, and let
$(T,\sigma)$ be a $\Sigma$\dash tree.
Let $\varphi$ be a bijection from the finite set $E$ into $L(T)$.
Let $\mcal{S}=\{Y_{i}\}_{i\in I}$ be a family of subsets of $E$.
Now we define $\enc(T,\sigma,\varphi,\mcal{S})$ to be a
$(\Sigma \cup \Sigma\times \{0,1\}^{I})$\dash tree with $T$ as
its underlying tree.
If $I$ is empty, then we simply set $\enc(T,\sigma,\varphi,\mcal{S})$ to 
be $(T,\sigma)$.
Now we assume $I$ is non-empty.
If $v$ is a non-leaf vertex of $T$, then it receives the label
$\sigma(v)$ in $\enc(T,\sigma,\varphi,\mcal{S})$.
However, if $v$ is a leaf, then it receives a label
$(\sigma(v), s)$, where $s$ is the function from $I$ to $\{0,1\}$
taking $i$ to $1$ if and only if $\varphi^{-1}(v)$ is in $Y_{i}$.
We think of the label on the leaf $v$ as containing a character
from the alphabet $\Sigma$, as well as a binary string
where each bit of the string encodes whether or not the
corresponding element $\varphi^{-1}(v)\in E$ is in a set $Y_{i}$.

We say that a tree automaton $A$ is \emph{$I$\dash ary} if $I$ is a
finite set of positive integers, the alphabet of $A$
is $\Sigma\cup \Sigma\times \{0,1\}^{I}$, and every image of
$\delta_{0}$ is in $\Sigma\times \{0,1\}^{I}$, for some
finite set $\Sigma$.
Under these circumstances, we blur the terminology by saying
that $\Sigma$ itself is the alphabet of the automaton.

\begin{definition}
\label{option}
Let $\Sigma$ be a finite set, and let
$A$ be an $\{i\}$\dash ary tree automaton with alphabet $\Sigma$.
Let $(T,\sigma)$ be a $\Sigma$\dash tree, and let
$\varphi$ be a bijection from the finite set $E$ into $L(T)$.
We define the set-system $M(A,T,\sigma,\varphi)$
as follows:
\[
M(A,T,\sigma,\varphi)
=
(E,\{Y_{i}\subseteq E\colon A\ \text{accepts}\
\enc(T,\sigma,\varphi,\{Y_{i}\})\}).
\]
\end{definition}

So the ground set of $M(A,T,\sigma,\varphi)$ is in bijection with the
leaves of $T$ and the independent sets are exactly the subsets that are accepted by $A$,
where a subset is encoded by applying $0$-$1$ labels to the leaves.

Now we are ready to give our second main definition.

\begin{definition}
\label{dancer}
Let \mcal{M} be a class of set-systems.
Assume that $A$ is an $\{i\}$\dash ary tree automaton
with alphabet $\Sigma$.
Assume also that for any $M=(E,\mcal{I})$ in \mcal{M}, there is a
$\Sigma$\dash tree $(T_{M},\sigma_{M})$, and a bijection
$\varphi_{M}\colon E\to L(T_{M})$ having the property that
$M=M(A,T_{M},\sigma_{M},\varphi_{M})$.
In this case we say that \mcal{M} is \emph{automatic}.
\end{definition}

Note that any subclass of an automatic class is also automatic.
We say that $(T_{M},\sigma_{M})$ from \Cref{dancer} is a \emph{parse tree}
for $M$ (relative to the automaton $A$).

\begin{definition}
\label{yodler}
Let \mcal{M} be a class of matroids.
We say that \mcal{M} is \emph{automatic} if the class of
set-systems $\{(E(M),\mcal{I}(M))\colon M\in\mcal{M}\}$ is
automatic.
\end{definition}

Thus a class of matroids is automatic if there is an automaton
that acts as follows:
for each matroid $M$ in the class, there is a parse tree
$(T_{M},\sigma_{M})$, and a bijection $\varphi_{M}$ from the ground set
of $M$ to the leaves, such that when the leaf labels encode the set
$Y_{i}\subseteq E(M)$, the automaton accepts if and only if $Y_{i}$ is independent.
In other words, there is an automaton that will simulate an
independence oracle on an appropriately chosen parse tree
for any matroid in the class.

The next \namecref{family} says that if there is an automaton that simulates
an independence oracle, then there is an automaton that will test any
\cmso\ formula.
The ideas in the proof appear to have originated with Kleene \cite{Kle56}.

\begin{lemma}
\label{family}
Let $A'$ be an $\{i\}$\dash ary tree automaton with alphabet $\Sigma$.
Let $\psi$ be a formula in \cmso\ with free variables
$\{X_{j}\}_{j\in I}$.
There is an $I$\dash ary tree automaton $A$ with alphabet $\Sigma$,
such that for every $\Sigma$\dash tree $(T,\sigma)$,
every bijection, $\varphi$, from a finite set $E$ into $L(T)$,
and every family $\mcal{S}=\{Y_{j}\}_{j\in I}$ of subsets of $E$,
the automaton $A$ accepts $\enc(T,\sigma,\varphi,\mcal{S})$
if and only if the set-system $M(A',T,\sigma,\varphi)$ satisfies $\psi$ under
the interpretation taking $X_{j}$ to $Y_{j}$ for each $j\in I$.
\end{lemma}

When we say that $A$ \emph{decides} $\psi$, we mean that
$A$ accepts $\enc(T,\sigma,\varphi,\mcal{S})$
if and only if $M(A',T,\sigma,\varphi)$ satisfies $\psi$ under
the interpretation taking $X_{j}$ to $Y_{j}$, for any
$T$, $\sigma$, and $\varphi$.

\begin{remark}
\label{aether}
If \mcal{M} is an automatic class,
then by definition, for each $M\in \mcal{M}$, we can choose
$T$, $\sigma$, and $\varphi$ so that
$M(A',T,\sigma,\varphi)$ is $M$.
Therefore \Cref{family} will provide us with a way to test whether
$M$ satisfies $\psi$: we simply run $A$ on
the appropriately labelled tree.
\end{remark}

\begin{proof}[Proof of \textup{\Cref{family}}.]
We prove the \namecref{family} by induction on the number of steps used
to construct the formula $\psi$.
Start by assuming that $\psi$ is atomic.
Assume that $\psi$ is $\mathrm{Ind}(X_{j})$.
Then the result follows from the definitions by setting $A$ to be $A'$.

Next we assume that $\psi$ is the atomic formula $X_{j}\subseteq X_{k}$.
We let the state space of $A$ be $\{\checkmark,\times\}$, and let
$\checkmark$ be the only accepting state.
Define $\delta_{0}$ so that for any $\alpha\in\Sigma$ and
any function $s\in\{0,1\}^{\{j,k\}}$, the image
$\delta_{0}(\alpha,s)$ is $\{\times\}$ if $(s(j),s(k))=(1,0)$, and
otherwise $\delta_{0}(\alpha,s)$ is $\{\checkmark\}$.
We define $\delta_{2}$ so that for any $\alpha\in\Sigma$,
\[
\delta_{2}(\alpha,\times,\times)=
\delta_{2}(\alpha,\times,\checkmark)=
\delta_{2}(\alpha,\checkmark,\times)=
\{\times\}
\]
and $\delta_{2}(\alpha,\checkmark,\checkmark)=\{\checkmark\}$.
Note that as $A$ processes the tree, it assigns $\times$ to a leaf
if and only if the corresponding element of $E$ is in $Y_{j}$ but not
$Y_{k}$.
If any leaf is assigned $\times$, then this state is propagated towards the
root.
Thus $A$ decides the formula $X_{j}\subseteq X_{k}$, as desired.

Next we will assume that $\psi$ is the atomic formula $|X_{j}|_{p,q}$.
We set the state space of $A$ to be $\{0,1,\ldots, q-1\}$, and we
let $p$ be the only accepting state.
For any $\alpha\in \Sigma$ and any $s\in\{0,1\}^{j}$, we set
$\delta_{0}(\alpha, s)$ to be $s(j)$.
Now for any $\alpha\in \Sigma$ and any $x,y\in\{0,1,\ldots,q-1\}$, we set
$\delta_{2}(\alpha,x,y)$ to be the residue of $x+y$ modulo $q$.
It is clear that $A$ decides $|X_{j}|_{p,q}$.

We may now assume that $\psi$
is not atomic.
Assume that $\psi$ is a negation, $\neg\phi$.
Note that the free variables of $\phi$ are $\{X_{j}\}_{j\in I}$.
By induction, there is an automaton, $A_{\phi}$, that accepts
$\enc(T,\sigma,\varphi,\{Y_{j}\}_{j\in I})$
if and only if
$M(A',T,\sigma,\varphi)$ satisfies $\phi$ under the interpretation
taking each $X_{j}$ to $Y_{j}$.
By \Cref{folder}, we can assume that $A_{\phi}$
is deterministic.
Now we produce $A$ by modifying $A_{\phi}$ so that a
state is accepting in $A$ exactly when it is not accepting in $A_{\phi}$.
Then $A$ decides $\neg\phi$.

Next we assume that $\psi$ is a conjunction,
$\phi_{1}\land \phi_{2}$.
For $z=1,2$, let $I_{z}$ be the set of free variables in $\phi_{z}$.
Thus $I=I_{1}\cup I_{2}$.
Inductively, there are automata $A_{1}$ and $A_{2}$
that decide $\phi_{1}$ and $\phi_{2}$.
For $z=1,2$, assume that $A_{i}$ is the automaton
\[(\Sigma\cup \Sigma\times \{0,1\}^{I_{z}}, Q^{z}, F^{z}, \delta_{0}^{z},\delta_{2}^{z}).\]
The idea of this proof is quite simple: we let
$A$ run $A_{1}$ and $A_{2}$ in parallel, and accept if and only
if both $A_{1}$ and $A_{2}$ accept.
To that end, we set $Q$ to be $Q^{1}\times Q^{2}$, and set
$F$ to be $F^{1}\times F^{2}$.
If $s$ is a function in $\{0,1\}^{I}$, then $s\!\restriction_{I_{z}}$
is the restriction of $s$ to $I_{z}$.
Now we define $\delta_{0}$ so that it takes
$(\alpha,s)$ to
\[
\delta_{0}^{1}(\alpha,s\!\restriction_{I_{1}})
\times \delta_{0}^{2}(\alpha,s\!\restriction_{I_{2}})
\]
for any $\alpha\in \Sigma$ and any $s\in\{0,1\}^{I}$.
We similarly define $\delta_{2}$ so that
$\delta_{2}(\alpha,(q_{L}^{1},q_{L}^{2}),(q_{R}^{1},q_{R}^{2}))$ is
\[
\delta_{2}^{1}(\alpha,q_{L}^{1},q_{R}^{1})\times
\delta_{2}^{2}(\alpha,q_{L}^{2},q_{R}^{2}).
\]
It is easy to see that $A$ acts as we desire, and therefore
decides $\psi$.

Finally, we must assume that $\psi$ is $\exists X_{j}\phi$,
where the free variables of $\phi$ are $\{X_{k}\}_{k\in I\cup \{j\}}$
and $j$ is not in $I$.
By induction, we can assume that the automaton
\[
A_{\phi}=(\Sigma\cup\Sigma\times \{0,1\}^{I\cup j},Q^{\phi},
F^{\phi},\delta_{0}^{\phi},\delta_{2}^{\phi})
\]
decides $\phi$.
For each $s\in\{0,1\}^{I}$, we set $s^{0}$ to be
the function in $\{0,1\}^{I\cup j}$ such that
$s^{0}\restriction_{I}=s$, and $s^{0}(j)=0$.
We similarly define $s^{1}\in\{0,1\}^{I\cup j}$ so that
$s^{1}\restriction_{I}=s$ and $s^{1}(j)=1$.
Now for each $\alpha\in\Sigma$ we set
\[
\delta_{0}(\alpha,s)=\delta_{0}^{\phi}(\alpha,s^{0})\cup
\delta_{0}^{\phi}(\alpha,s^{1}).
\]
Thus $\delta_{0}$ sends $(\alpha,s)$ to the set of states
that could be applied by $A_{\phi}$ to a leaf labelled
by $(\alpha,s')$, where $s'$ extends the domain of $s$ to include $j$.
We define $\delta_{2}(\alpha,q_{L},q_{R})$
to be $\delta_{2}^{\phi}(\alpha,q_{L},q_{R})$ when $\alpha$ is in $\Sigma$.
We define the state space and the accepting states of
$A$ to be exactly those of $A_{\phi}$.
We must now show that $A$ decides $\exists X_{j}\phi$.
We let $(T,\sigma)$ be an arbitrary $\Sigma$\dash tree, and
we let $\varphi$ be a bijection from the finite set $E$ into $L(T)$.

Assume that $M(A',T,\sigma,\varphi)$ satisfies $\exists X_{j}\phi$
under the interpretation that takes $X_{k}$ to $Y_{k}\subseteq E$
for each $k\in I$.
Then there is a subset $Y_{j}\subseteq E$ such that
$M(A',T,\sigma,\varphi)$ satisfies $\phi$ under the interpretation
that takes $X_{k}$ to $Y_{k}$ for all $k\in I\cup j$.
Let $\mcal{S}_{\phi}$ be $\{Y_{k}\}_{k\in I\cup j}$.
By induction, $A_{\phi}$ accepts
$\enc(T,\sigma,\varphi,\mcal{S}_{\phi})$.
Let $r_{\phi}$ be the run of $A_{\phi}$ on
$\enc(T,\sigma,\varphi,\mcal{S}_{\phi})$.
Then $r_{\phi}(t)$ contains a state in $F^{\phi}$, where $t$ is the
root of $T$.
Let $\mcal{S}=\{Y_{k}\}_{k\in I}$, and let
$r$ be the run of $A$ on
$\enc(T,\sigma,\varphi,\mcal{S})$.
It is easy to inductively prove
that $r(v)\supseteq r_{\phi}(v)$ for every vertex $v$.
Therefore $r(t)$ contains an accepting state, so
$A$ accepts $\enc(T,\sigma,\varphi,\mcal{S})$.

For the converse, assume that $A$ accepts
$\enc(T,\sigma,\varphi,\mcal{S})$,
where $\mcal{S}=\{Y_{k}\}_{k\in I}$ is a family of subsets of $E$.
Let $r$ be the run of $A$ on
$\enc(T,\sigma,\varphi,\mcal{S})$.
We recursively nominate a state $q(v)$ chosen from
$r(v)$, for each vertex $v$.
Since $A$ accepts, there is an accepting state
in $r(t)$.
We define $q(t)$ to be this accepting state.
Now assume that $q(v)$ is defined, and that
the children of $v$ are $v_{L}$ and $v_{R}$.
Then there are states $q_{L}\in r(v_{L})$
and $q_{R}\in r(v_{R})$ such that
$\delta_{2}(\sigma(v),q_{L},q_{R})$ contains $q(v)$.
We choose $q(v_{L})$ to be $q_{L}$ and
$q(v_{R})$ to be $q_{R}$.
Thus we have defined $q(v)$ for each vertex $v$.

We will now define a set $Y_{j}\subseteq E$.
Let $v$ be an arbitrary leaf.
We describe a method for deciding if $\varphi^{-1}(v)$ is in
$Y_{j}$.
Let $s\in \{0,1\}^{I}$ be the function that records
whether $\varphi^{-1}(v)$ is in $Y_{k}$, for $k\in I$.
Thus $\delta_{0}(\sigma(v),s)$ includes $q(v)$.
Now
\[
\delta_{0}(\sigma(v),s)
=\delta_{0}^{\phi}(\sigma,s^{0})\cup
\delta_{0}^{\phi}(\sigma,s^{1})
\]
If $q(v)$ is in $\delta_{0}^{\phi}(\sigma,s^{0})$,
we declare $\varphi^{-1}(v)$ not to be in $Y_{j}$.
Otherwise we declare $\varphi^{-1}(v)$ to be in $Y_{j}$.

Let $\mcal{S}_{\phi}$ be the family $\{Y_{k}\}_{k\in I\cup j}$.
Let $r_{\phi}$ be the run of $A_{\phi}$ on
$\enc(T,\sigma,\varphi,\mcal{S}_{\phi})$.
It is easy to prove by induction that
$r_{\phi}(v)$ contains $q(v)$ for every vertex $v$.
Therefore $r_{\phi}(t)$ contains an accepting state,
so $A_{\phi}$ accepts
$\enc(T,\sigma,\varphi,\mcal{S}_{\phi})$.
By induction, this means that $M(A',T,\sigma,\varphi)$
satisfies $\phi$
under the interpretation taking each $X_{k}$ to
$Y_{k}$ for $k\in I\cup j$.
Hence $\exists X_{j}\phi$ is satisfied by the
interpretation taking $X_{k}$ to $Y_{k}$ for $k\in I$.
This completes the proof that $A$ decides
$\psi=\exists X_{j}\phi$, and hence the proof of the
\namecref{family}.
\end{proof}

\section{Characterising automatic classes}
\label{mainproof}

Now we can prove \Cref{mtheorem1}.
We split the proof into two lemmas.

\begin{lemma}
\label{redcap}
Let \mcal{M} be a class of set-systems.
If \mcal{M} is automatic, then it has bounded decomposition-width.
\end{lemma}

\begin{proof}
Since \mcal{M} is automatic, we can let $A$ be an $\{i\}$\dash ary tree automaton
with alphabet $\Sigma$ and state space $Q$ such that for every $M=(E,\mcal{I})$
in \mcal{M}, there is a $\Sigma$\dash tree $(T_{M},\sigma_{M})$ and a
bijection $\varphi_{M}\colon E\to L(T_{M})$ having the
property that $A$ accepts $\enc(T_{M},\sigma_{M},\varphi_{M},\{Y_{i}\})$
if and only if $Y_{i}$ is in \mcal{I}, for any $Y_{i}\subseteq E$.
By applying \Cref{folder}, we can assume that $A$
is deterministic.

Let $M=(E,\mcal{I})$ be an arbitrary set-system in \mcal{M}.
Let $e$ be an arbitrary edge in $T_{M}$, and assume $e$ is incident
with the vertices $u$ and $v$.
The subgraph of $T_{M}$ obtained by deleting $e$ contains
two components, $T_{u}$ and $T_{v}$, containing $u$ and $v$
respectively.
By relabelling as necessary, we will assume that $T_{v}$
contains the root $t$.
We let $U_{e}$ be the set containing elements $z\in E(M)$
such that the path from $\varphi_{M}(z)$ to $t$ contains the edge $e$.
Let $V_{e}$ be $E-U_{e}$.
We will show that the relation $\sim_{U_{e}}$ induces
at most $|Q|$ equivalence classes.
\Cref{shadow} will then imply that \mcal{M} has decomposition-width
at most $2^{|Q|}$.
(Although $(T_{M},\varphi_{M})$ is not a decomposition of $M$, it can easily
be turned into one by contracting an edge incident with the root,
and then forgetting the distinction between left and right children.)

Let $Y$ and $Y'$ be arbitrary subsets of $U_{e}$.
Let $r_{1}$ and $r'_{1}$ be the runs of $A$ on
$\enc(T_{M},\sigma_{M},\varphi_{M},\{Y\})$ and
$\enc(T_{M},\sigma_{M},\varphi_{M},\{Y'\})$
respectively.
We declare $Y$ and $Y'$ to be equivalent if and only if
these runs apply the same singleton set to $u$; that is, if $r_{1}(u)=r'_{1}(u)$.
It is clear that this is an equivalence relation on subsets of $U_{e}$
with at most $|Q|$ equivalence classes, so it remains to
show that this equivalence relation refines $\sim_{U_{e}}$.
Assume that $Y$ and $Y'$ are equivalent subsets, and let
$Z$ be an arbitrary subset of $V_{e}$.
Let $r_{2}$ and $r_{2}'$ be the runs of $A$ on
$\enc(T_{M},\sigma_{M},\varphi_{M},\{Y\cup Z\})$ and
$\enc(T_{M},\sigma_{M},\varphi_{M},\{Y'\cup Z\})$.
Any leaf in $T_{u}$ receives the same label in both
$\enc(T_{M},\sigma_{M},\varphi_{M},\{Y\})$ and
$\enc(T_{M},\sigma_{M},\varphi_{M},\{Y\cup Z\})$.
Now it is easy to prove by induction that
$r_{1}(w)=r_{2}(w)$ for all vertices $w$ in $T_{u}$.
Similarly, $r_{1}'(w)=r_{2}'(w)$ for all such $w$.
In particular, $r_{2}(u)=r_{1}(u)=r_{1}'(u)=r_{2}'(u)$, where
the middle equality is because of the equivalence of $Y$ and $Y'$.
Using the fact that $r_{2}(u)=r_{2}'(u)$, we can 
prove by induction that $r_{2}(w)=r_{2}'(w)$ for all vertices $w$ in $T_{v}$.
In particular, $r_{2}(t)=r_{2}'(t)$, so $A$ accepts
$\enc(T_{M},\sigma_{M},\varphi_{M},\{Y\cup Z\})$
if and only if it accepts
$\enc(T_{M},\sigma_{M},\varphi_{M},\{Y'\cup Z\})$.
This implies that $Y\cup Z$ is in \mcal{I} if and only if $Y'\cup Z$ is.
Thus $\sim_{U_{e}}$ has at most $|Q|$ classes, as desired.
\end{proof}

The other direction is known to Kr\'{a}l \cite{Kra12}
and to Strozecki \cite{Str10,Str11}.

\begin{lemma}
\label{shader}
Let \mcal{M} be a class of set-systems.
If \mcal{M} has bounded decomposition-width, then it is
automatic.
\end{lemma}

\begin{proof}
Let $K$ be an integer such that $\dw(M)\leq K$
for all members $M\in \mcal{M}$.
Thus, any member \mcal{M} has a
decomposition such that each displayed set
contains at most $K$ equivalence classes.
We construct a tree automaton, $A$, that
decides the formula $\ind(X_{i})$.
The set of states of $A$ is
$Q=\{\indep,\dep,q_{1},\ldots, q_{K}\}$.

Let $M=(E,\mcal{I})$ be an arbitrary set-system in \mcal{M}, and
let $(T,\varphi)$ be a decomposition of $M$, where $\sim_{U}$
has at most $K$ equivalence classes for any set
$U$ displayed by an edge of $T$.
We start by showing how to construct the parse tree
$(T_{M},\sigma_{M})$ by modifying $T$.
First, we arbitrarily choose an edge of $T$, and subdivide
it with the new vertex $t$, where $t$ will be the root of $T_{M}$.
For each non-leaf vertex of $T$, we make an arbitrary
decision as to which of its children is the left child, and which is
the right.
This describes the tree $T_{M}$.
The bijection $\varphi_{M}$ is set to be identical to $\varphi$.

For each edge $e$, let $U_{e}$ be the set of elements
$z\in E$ such that the path from $\varphi_{M}(z)$ to $t$
contains the edge $e$.
Then $\sim_{U_{e}}$ induces at most $K$
equivalence classes.
Let $\ell_{e}$ be some function from
the subsets of $U_{e}$ into
$\{q_{1},\ldots, q_{K}\}$
such that $\ell_{e}(X)=\ell_{e}(X')$ implies
$X\sim_{U_{e}} X'$.
We think of $\ell_{e}$ as applying labels
to the equivalence classes of $\sim_{U_{e}}$.
(Although we allow the possibility that equivalent subsets under
$\sim_{U_{e}}$ receive different labels under $\ell_{e}$.
In other words, the equivalence relation induced by $\ell_{e}$
refines $\sim_{U_{e}}$.)
For each $q_{j}$ in the image $\image(\ell_{e})$, we arbitrarily
choose a representative subset $\Rep_{e}(q_{j})\subseteq U_{e}$
such that $\ell_{e}(\Rep_{e}(q_{j}))=q_{j}$.

Next we describe the function $\sigma_{M}$, which labels each vertex
of $T_{M}$ with a function.
Let $u$ be a leaf of $T_{M}$.
Then $\sigma_{M}(u)$ is a function, $f$, whose domain is $\{0,1\}$.
In the case that $u$ is also the root of $T_{M}$,
we set $f(0)$ to be the symbol \indep\ if $\emptyset$ is in \mcal{I},
and otherwise we set $f(0)$ to be the symbol \dep.
Similarly, $f(1)=\indep$ if $\{\varphi_{M}^{-1}(u)\}$ is in \mcal{I}, and
otherwise $f(1)=\dep$.
Now assume that $u$ is a non-root leaf, and let $e$
be the edge incident with $u$.
Then $f(0)$ is the label $\ell_{e}(\emptyset)$, and
$f(1)$ is $\ell_{e}(\{\varphi_{M}^{-1}(u)\})$.

Now let $u$ be a non-leaf vertex.
Let $e_{L}$ and $e_{R}$ be the edges joining
$u$ to its children.
Then $\sigma_{M}(u)$ is a function $f$ and the domain of $f$ is
$\image(\ell_{e_{L}})\times \image(\ell_{e_{R}})$.
Let $(q_{j},q_{k})$ be in $\image(\ell_{e_{L}})\times \image(\ell_{e_{R}})$,
and assume $X_{j}\subseteq U_{e_{L}}$ is the representative $\Rep_{e_{L}}(q_{j})$,
while $X_{k}$ is $\Rep_{e_{R}}(q_{k})$.
Assume that $u$ is not the root, and let $e$ be the
first edge in the path from $u$ to $t$.
Then $f(q_{j},q_{k})$ is $\ell_{e}(X_{j}\cup X_{k})$, for each
such $(q_{j},q_{k})$.
Next assume that $u$ is the root.
Then $f(q_{j},q_{k})$ is \indep\ if $X_{j}\cup X_{k}\in\mcal{I}$,
and otherwise $f(q_{j},q_{k})=\dep$.

Now we have completed our description of $\sigma_{M}$, which labels the
vertices of $T_{M}$ with functions.
Therefore $(T_{M},\sigma_{M})$ is a $\Sigma$\dash tree,
where $\Sigma$ is the alphabet of partial functions
from $\{0,1\}\cup (2^{\{q_{1},\ldots, q_{K}\}}\times 2^{\{q_{1},\ldots, q_{K}\}})$
into $\{\indep,\dep,q_{1},\ldots, q_{K}\}$.

Our next task is to describe the automaton, $A$.
As we have said, the state space is
$Q=\{\indep,\dep,q_{1},\ldots, q_{K}\}$.
The alphabet is $\Sigma\cup \Sigma\times \{0,1\}^{\{i\}}$, where
$\Sigma$ is the set of partial functions we described in the previous
paragraph.
The only accepting state is $\indep$.
To define the transition rule $\delta_{0}$, we consider the input
$(f,s)$, where $f$ is a function from $\{0,1\}$ into $Q$, and
$s$ is a function in $\{0,1\}^{\{i\}}$.
Then we define $\delta_{0}(f,s)$ to be $\{f(s(i))\}$.
Now we consider the transition rule $\delta_{2}$.
Let $f$ be a function whose domain is a member of
$2^{\{q_{1},\ldots, q_{K}\}}\times 2^{\{q_{1},\ldots, q_{K}\}}$.
Assume that $(q_{i},q_{j})$ is in the domain of $f$.
Then $\delta_{2}(f,q_{i},q_{j})$ is defined to be
$\{f(q_{i},q_{j})\}$.
This completes our description of the automaton $A$.
Note that it is deterministic.

\begin{claim}
\label{pocket}
Let $Y_{i}$ be a subset of $E$.
Let $u$ be a non-root vertex of $T_{M}$, and let $e$ be the
first edge on the path from $u$ to $t$.
Let $q$ be the state applied to $u$ by the 
run of $A$ on $\enc(T_{M},\sigma_{M},\varphi_{M},\{Y_{i}\})$.
Then $(Y_{i}\cap U_{e})\sim_{U_{e}} \Rep_{e}(q)$.
\end{claim}

\begin{proof}
Assume that $u$ has been chosen so that it is as far away from
$t$ as possible, subject to the constraint that the
\namecref{pocket} fails for $u$.
Let $f$ be the function applied to $u$ by the labelling $\sigma_{M}$.

First assume that $u$ is a leaf, so that $U_{e}=\{\varphi_{M}^{-1}(u)\}$.
Then $u$ receives the label $(f,s)$ in
$\enc(T_{M},\sigma_{M},\varphi_{M},\{Y_{i}\})$, where $s(i)$ is $1$
if $\varphi_{M}^{-1}(u)\in Y_{i}$, and is $0$ otherwise.
The construction of $A$ means that $q=f(s(i))$.
If $Y_{i}\cap U_{e}=\emptyset$, then
$q=f(0)=\ell_{e}(\emptyset)$.
Now $\ell_{e}(\Rep_{e}(q))=q$, by definition, so
$\Rep_{e}(q) \sim_{U_{e}}\emptyset$, by the nature of the function $\ell_{e}$.
Therefore $(Y_{i}\cap U_{e})\sim_{U_{e}}\Rep_{e}(q)$, as desired.
The other possibility is that $Y_{i}\cap U_{e}=U_{e}=\{\varphi_{M}^{-1}(u)\}$.
In this case $q=f(1)=\ell_{e}(U_{e})$.
Again $\Rep_{e}(q)\sim_{U_{e}} U_{e}$, and hence
$(Y_{i}\cap U_{e}) \sim_{U_{e}} \Rep_{e}(q)$.

Now we must assume that $u$ is not a leaf, so that $u$ is joined
to its children, $u_{L}$ and $u_{R}$, by the edges
$e_{L}$ and $e_{R}$.
Assume that $u$ receives the label $f$ in
$\enc(T_{M},\sigma_{M},\varphi_{M},\{Y_{i}\})$.
Let $q_{L}$ and $q_{R}$ be the states applied to $u_{L}$ and $u_{R}$
by the run of $A$ on $\enc(T_{M},\sigma_{M},\varphi_{M},\{Y_{i}\})$.
Our inductive assumption on $u$ means that
$(Y_{i}\cap U_{e_{L}})\sim_{U_{e_{L}}} \Rep_{e_{L}}(q_{L})$ and
$(Y_{i}\cap U_{e_{R}})\sim_{U_{e_{R}}} \Rep_{e_{R}}(q_{R})$.
Let $X_{j}$ be $\Rep_{e_{L}}(q_{L})$ and use $X_{k}$ to denote
$\Rep_{e_{R}}(q_{R})$.
Now \Cref{icecap} implies that
$(Y_{i}\cap U_{e}) = (Y_{i}\cap U_{e_{L}})\cup (Y_{i}\cap U_{e_{R}})$
is equivalent to $X_{j}\cup X_{k}$ under $\sim_{U_{e}}$.
The construction of $f$ and $A$ means that $q=\ell_{e}(X_{j}\cup X_{k})$.
Obviously $\ell_{e}(\Rep_{e}(q))=q$, so the nature of the function $\ell_{e}$ implies
$(X_{j}\cup X_{k})\sim_{U_{e}} \Rep_{e}(q)$.
Now we see that $(Y_{i}\cap U_{e})\sim_{U_{e}} \Rep_{e}(q)$, so $u$ fails to provide
a counterexample after all.
\end{proof}

If the root, $t$, is a leaf, then $A$ applies \indep\ to $t$ if and only
if $Y_{i}\cap \{\varphi_{M}^{-1}(t)\}=Y_{i}$ is in \mcal{I}.
Assume that $t$ is not a leaf, and that the edges $e_{L}$ and $e_{R}$
join $t$ to its children, $u_{L}$ and $u_{R}$.
Let $q_{L}$ and $q_{R}$ be the states applied to $u_{L}$ and $u_{R}$.
Let $X_{j}$ be $\Rep_{e_{L}}(q_{L})$, and let $X_{k}$ be
$\Rep_{e_{R}}(q_{R})$.
Then $(Y_{i}\cap U_{e_{L}})\sim_{U_{e_{L}}} X_{j}$ and
$(Y_{i}\cap U_{e_{R}})\sim_{U_{e_{R}}} X_{k}$, by \Cref{pocket}.
If we apply \Cref{icecap} with $U=U_{e_{L}}$ and $V=U_{e_{R}}$, we see that both of
$Y_{i}=(Y_{i}\cap U_{e_{L}})\cup (Y_{i}\cap U_{e_{R}})$ and
$X_{j}\cup X_{k}$ belong to \mcal{I}, or neither does.
In the former case, $A$ applies \indep\ to $t$ during its
run on $\enc(T_{M},\sigma_{M},\varphi_{M},\{Y_{i}\})$,
and hence accepts.
In the latter case, $A$ applies \dep, and does not accept.
Therefore $A$ decides $\ind(Y_{i})$, exactly as we want.
\end{proof}

Recall that a class of matroids is pigeonhole if every subclass
with bounded branch-width also has bounded decomposition-width.
Now we can deduce the following (perhaps not obvious) fact.

\begin{corollary}
\label{dynamo}
Let \mcal{M} be a pigeonhole class of matroids.
Then $\{M^{*}\colon M\in \mcal{M}\}$ is pigeonhole.
\end{corollary}

\begin{proof}
Assume that \mcal{M} is pigeonhole.
For every positive integer, $\lambda$, there is an integer
$\rho(\lambda)$ such that any matroid in \mcal{M} with branch-width at
most $\lambda$ has decomposition-width at most $\rho(\lambda)$.

Let $\lambda$ be an arbitrary positive integer.
Let $\mcal{M}_{\lambda}$ be the class of matroid in \mcal{M} with branch-width at
most $\lambda$.
As $\mcal{M}_{\lambda}$ has bounded decomposition-width,
\Cref{shader} implies that it is an automatic class.
Let $A'$ be an $\{i\}$\dash ary automaton such that
for every matroid $M\in \mcal{M}_{\lambda}$,
there is a parse tree $(T_{M},\sigma_{M})$ and a
bijection $\varphi_{M}\colon E(M)\to L(T_{M})$ such that
$M=M(A',T_{M},\sigma_{M},\varphi_{M})$.

The predicate
\[
\formula{Basis}(X_{2})=\ind(X_{2})\land \forall X_{3}
((\ind(X_{3})\land X_{2}\subseteq X_{3})\to X_{3}\subseteq X_{2})
\]
is satisfied exactly by interpretations that take $X_{2}$ to a basis
of a matroid.
Similarly,
\[
\formula{Coind}(X_{1})=\exists X_{2} (\formula{Basis}(X_{2})\land \neg\exists X_{4}
(\sing(X_{4})\land X_{4}\subseteq X_{1}\land X_{4}\subseteq X_{2})
\]
is satisfied exactly by the interpretations that take
$X_{1}$ to coindependent sets.
Now \Cref{family} implies that there is an automaton, $A$,
that accepts $\enc(T_{M},\sigma_{M},\varphi_{M},\{Y_{i}\})$
if and only if $Y_{i}$ is coindependent in $M$, for each $M\in \mcal{M}_{\lambda}$.
Therefore $M(A,T_{M},\sigma_{M},\varphi_{M})=M^{*}$, so
this establishes that $\{M^{*}\colon M\in\mcal{M}_{\lambda}\}$ is
an automatic class of matroids.
\Cref{redcap} implies there is an integer $\rho^{*}(\lambda)$ such that
$\dw(M^{*})\leq \rho^{*}(\lambda)$ whenever $M$ is in $\mcal{M}_{\lambda}$.

The branch-width of a matroid is equal to the branch-width of its dual
\cite[Proposition 14.2.3]{Oxl11}.
Hence
\[
\{M^{*}\colon M\in \mcal{M},\ \bw(M^{*})\leq \lambda\}=\{M^{*}\colon M\in \mcal{M}_{\lambda}\}.
\]
We have just shown that any matroid in this class has decomposition-width
at most $\rho^{*}(\lambda)$, and this establishes the result.
\end{proof}

We do not know of a proof of \Cref{dynamo} that does not rely on \Cref{mtheorem1}.
We do not know if the dual of a strongly pigeonhole class must
be strongly pigeonhole, but we conjecture that this is the case.

\begin{conjecture}
\label{spring}
Let \mcal{M} be a strongly pigeonhole class of matroids.
Then $\{M^{*}\colon M\in\mcal{M}\}$ is strongly pigeonhole.
\end{conjecture}

\section{Complexity theory}
\label{complexity}

In this section, we discuss complexity theoretical applications
of tree automata.
We start with a simple observation.

\begin{proposition}
\label{ermine}
Let $\psi$ be any sentence in \cmso.
Let \mcal{M} be an automatic class of set-systems.
There exists a Turing Machine which will take as input a parse
tree for any set system $M=(E,\mcal{I})\in\mcal{M}$
and then test whether or not $M$ satisfies $\psi$.
The running time is $O(n)$, where $n=|E|$.
\end{proposition}

\begin{proof}
Since \mcal{M} is automatic, we can assume that $A'$ is an $\{i\}$\dash ary tree
automaton with alphabet $\Sigma$, and for any
$M=(E,\mcal{I})\in \mcal{M}$ there is a parse tree $(T_{M},\sigma_{M})$ of $M$ relative to $A'$.
So there is a bijection $\varphi_{M}\colon E\to L(T_{M})$ such that
$A'$ accepts $\enc(T_{M},\sigma_{M},\varphi_{M},\{Y_{i}\})$ if and only if
$Y_{i}\in \mcal{I}$.
The proof of \Cref{family} is constructive, and shows us how to build an
automaton, $A$, which will accept $\enc(T_{M},\sigma_{M},\varphi_{M},\emptyset)$
if and only if $M$ satisfies $\psi$.
This construction is done during pre-processing, so it has no impact on the running time.
While $A$ processes $\enc(T_{M},\sigma_{M},\varphi_{M},\emptyset)$, the computation
that occurs at each node takes a constant amount of time.
So the running time of $A$ is proportional to the number of nodes.
This number is $2n-1$, so the result follows.
\end{proof}

Various models of matroid computation have been studied.
Here, we will concentrate on classes of matroids that
have compact descriptions.

\begin{definition}
\label{umpire}
Let \mcal{M} be a class of matroids.
A \emph{succinct representation} of \mcal{M} is a
relation, $\Delta$, from \mcal{M} into the set of finite binary strings.
We write $\Delta(M)$ to indicate any string in the image of $M\in \mcal{M}$.
We insist that there is a polynomial $p$ and a Turing Machine which will return an
answer to the question ``Is $X$ independent in $M$?" in time bounded by $p(|E(M)|)$.
Here the input is of the form $(\Delta(M),X)$, where $M\in\mcal{M}$ and $X$
is a subset of $E(M)$.
\end{definition}

Thus we insist that an independence oracle can
be efficiently simulated using the output of a succinct representation.
This constraint implies that $\Delta(M)$ and $\Delta(M')$ are disjoint when $M\ne M'$.
Note that the length $|\Delta(M)|$ can be no longer than $p(|E(M)|)$.
Descriptions of graphic or finite-field representable matroids
as graphs or matrices provide succinct representations.

\begin{proposition}
\label{lychee}
Let \mcal{M} be a class of matroids with succinct representation $\Delta$.
There is a Turing Machine which, for any integer $\lambda>0$,
will take as input any $\Delta(M)$ for $M\in \mcal{M}$ satisfying
$\bw(M)\leq \lambda$, and return a branch-decomposition
of $M$ with width at most $3\lambda+1$.
The running time is $O(8^{\lambda}n^{3.5}p(n))$, where $n=|E(M)|$
and $p$ is as in \textup{\Cref{umpire}}.
\end{proposition}

\begin{proof}
The proof of this \namecref{lychee} requires nothing more than an analysis
of the proof of \cite[Corollary 7.2]{OS06}, so we provide a sketch only.
Let $M=(E,\mcal{I})$ be a matroid in \mcal{M} with $\bw(M)\leq \lambda$.
A \emph{partial decomposition} of $M$ consists of a subcubic tree,
along with a partition of $E$ and a bijection from the blocks of this partition
into the leaf-set of $T$.
Each edge, $e$, of $T$ partitions $E$ into two sets, $U_{e}$ and $V_{e}$, and
the \emph{width} of $e$ is  $r_{M}(U_{e})+r_{M}(V_{e})-r(M)+1$.
We start with a partial decomposition containing a single block, and
successively partition blocks into two parts, until every block is a singleton set.
This process therefore takes $n-1$ steps.
At each step, we ensure that each edge has width at most $3\lambda+1$, so
at the end of the process, we will have the desired decomposition.
Assume that $U$ is a block in the partition with $|U|>1$.
Let $l$ be the leaf corresponding to $U$, and let $e$ be the edge
incident with $l$ (if $T$ is not a single vertex).
Let $V$ be $E-U$.
We inductively assume that the weight of $e$ is at most $3\lambda+1$.
If it is less than $3\lambda+1$, then we arbitrarily choose
an element $u\in U$, subdivide $e$ and join a new leaf to this new vertex.
We label the new leaf with $\{u\}$, and relabel the leaf corresponding to $U$ with
$U-\{u\}$.
Therefore we can assume that the width of $e$ is exactly $3\lambda+1$, and
hence $\lambda_{M}(U)=3\lambda$ (assuming $T$ has more than one vertex).

We use the greedy algorithm to find 
an arbitrary basis, $B$, of $M$ in $O(np(n))$ steps.
For any subset $X\subseteq U$, define $\lambda_{B}(X)$ to be
\[r_{M}(X\cup (B-V))+r_{M}(V\cup (B-X))-|B-X|-|B-V| + 1.\]
Then $\lambda_{B}(X)$ is the rank function of a matroid on the
ground set $U$ \cite[Propositions 4.1 and 7.1]{OS06}.
Let this matroid be $M_{B}$.
The rank of $M_{B}$ is $3\lambda+1$.
Finding the rank of $X\cup (B-V)$ in $M$ takes $O(np(n))$ steps,
using the greedy algorithm, and similarly for $V\cup (B-X)$ in $M$.
By again using the greedy algorithm, we can find a
basis, $D$, of $M_{B}$, in $O(np(n) + n^{2}p(n))$ steps.

Now we loop over all partitions of $D$ into an ordered pair of two sets,
$(D_{1},D_{2})$.
This takes $2^{3\lambda+1}$ steps.
We let $M_{1}$ and $M_{2}$ be $M/D_{1}\ba D_{2}$ and
$M\ba D_{1}/D_{2}$ respectively.
The ranks of $M_{1}$ and $M_{2}$ can be found in $O(np(n))$ time,
and it then takes $p(n)$ steps to test whether a subset is a basis of
$M_{1}$ or $M_{2}$.
Now it follows from \cite[Theorem 4.1]{Cun86} that we can
use an equivalent form of the matroid intersection algorithm
to find a set, $Z$, satisfying $D_{1}\subseteq Z\subseteq E-D_{2}$
that minimises $\lambda_{M}(Z)$.
Furthermore, this can be done in $O(np(n) + n^{2.5}p(n))$ steps.
If $\lambda_{M}(Z)+1\geq \min\{|D_{1}|,|D_{2}|\}$, then
$\bw(M) \geq |D|/3=\lambda+1/3$ and we
have a contradiction \cite[Theorem 5.1]{OS06}.
Therefore $\lambda_{M}(Z)+1<\min\{|D_{1}|,|D_{2}|\}$.
We subdivide $e$ and attach a leaf to the new vertex.
This leaf corresponds to the set $U \cap Z$, and we relabel $l$ with the set
$U-Z$.
(If $T$ has only one vertex, we simply create a tree with two vertices,
and label these with  $U \cap Z$ and $U-Z$.)

The proof of \cite[Theorem 5.2]{OS06} shows that the width of every edge
in the new decomposition is at most $3\lambda+1$, so we
can reiterate this process until we have a branch decomposition.
\end{proof}

We wish to develop efficient model-checking algorithms for
strongly pigeonhole matroid classes.
We have to strengthen this condition somewhat, by insisting not
only that there is a bound on the number of equivalence
classes, but that we can efficiently
compute the equivalence relation (or a refinement of it).

\begin{definition}
\label{yakuza}
Let \mcal{M} be a class of matroids with a succinct representation $\Delta$.
Assume there is a constant, $c$, and that for every integer, $\lambda>0$, there is an integer, $\pi(\lambda)$, and a
Turing Machine, $M_{\lambda}$, with the following properties:
$M_{\lambda}$ takes as input any tuple of the form $(\Delta(M),U,X,X')$,
where $M$ is in \mcal{M}, $U\subseteq E(M)$ satisfies $\lambda_{M}(U)\leq \lambda$,
and $X$ and $X'$ are subsets of $U$.
The machine $M_{\lambda}$ computes an equivalence
relation, $\approx_{U}$, on the subsets of $U$, so that
$M_{\lambda}$ accepts $(\Delta(M),U,X,X')$ if and only if
$X\approx_{U} X'$.
Furthermore,
\begin{enumerate}[label=\textup{(\roman*)}]
\item $X\approx_{U} X'$ implies $X\sim_{U} X'$,
\item the number of equivalence classes under
$\approx_{U}$ is at most $\pi(\lambda)$, and
\item $M_{\lambda}$ runs in time bounded by $O(\pi(\lambda)|E(M)|^{c})$.
\end{enumerate}
Under these circumstances, we say that \mcal{M} is
\emph{efficiently pigeonhole} (relative to $\Delta$).
\end{definition}

It follows immediately that if a class of matroids is
efficiently pigeonhole, then it is strongly pigeonhole.
We will later see that many natural classes 
are efficiently pigeonhole.

\begin{theorem}
\label{jibber}
Let \mcal{M} be a class of matroids with a succinct representation $\Delta$.
Assume that \mcal{M} is efficiently pigeonhole.
Let $\lambda$ be a positive integer.
There is a Turing Machine which
accepts as input any $\Delta(M)$ when $M\in\mcal{M}$
satisfies $\bw(M)\leq \lambda$, and returns a parse tree for $M$.
The running time is $O((8^{\lambda}n^{3.5}+\pi(3\lambda)^{2})p(n)+\pi(3\lambda)^{4}n^{c+1})$,
where $n=|E(M)|$, $p$ is as in \textup{\Cref{umpire}}, and $\pi$ and $c$ are as in \textup{\Cref{yakuza}}.
\end{theorem}

\begin{proof}
We start by applying \Cref{lychee} to obtain a branch-decomposition with width
at most $3\lambda+1$.
This construction takes $O(8^{\lambda}n^{3.5}p(n))$ steps.
Let $T$ be the tree underlying the branch-decomposition,
and let $\varphi$ be the bijection from $E(M)$ to the leaves of $T$.
We construct $T_{M}$ by subdividing an edge of $T$ with a root vertex, $t$,
and distinguishing between left and right children.
We let $\varphi_{M}$ be $\varphi$.
If $U$ is a set displayed by an edge of $T_{M}$, then
$\lambda_{M}(U)\leq 3\lambda$.
We let $K$ be $\pi(3\lambda)$, where $\pi$ is the function provided by
\Cref{yakuza}.

From this point we closely follow the proof of \Cref{shader}.
For each edge, $e$, in $T_{M}$, we perform the following procedure.
Let $u$ be the end-vertex of $e$ that is further from $t$ in $T_{M}$, and define
$U_{e}$ as in the proof of \Cref{shader}.
We will construct representative subsets, $\Rep_{e}(q)$, of $U_{e}$,
where $q$ is a label in $\{q_{1},\ldots, q_{K}\}$, in such a way that
distinct representative states are inequivalent under $\approx_{U_{e}}$.
At the same time, we will construct a function, $f$, which will be applied to
$u$ by the labelling function $\sigma_{M}$.

First assume that $u$ is a leaf, so that $U_{e}=\{\varphi_{M}^{-1}(u)\}$.
The domain of $f$ will be $\{0,1\}$.
As in \Cref{shader}, we must consider the case that
$u$ is also the root of $T_{M}$.
In this case, we set $f(0)$ to be \indep, and
set $f(1)$ to be \indep\ or \dep\ depending on whether $\{\varphi_{M}^{-1}(u)\}$ is
independent.
Now assume that $u$ is a non-root leaf.
Let $\emptyset$ be $\Rep_{e}(q_{1})$, and set $f(0)$ to be $q_{1}$.
In $O(Kn^{c})$ steps, we test whether
$\{\varphi_{M}^{-1}(u)\}\approx_{U_{e}} \emptyset$.
If so, then we set $f(1)$ to be $q_{1}$.
Assuming that $\{\varphi_{M}^{-1}(u)\}\not\approx_{U_{e}} \emptyset$,
we define $\Rep_{e}(q_{2})$ to be $U_{e}=\{\varphi_{M}^{-1}(u)\}$, and we
set $f(1)$ to be $q_{2}$.

Now assume that $u$ is not a leaf.
Let $e_{L}$ and $e_{R}$ be the edges joining $u$ to its children.
Recursively, we assume that $\Rep_{e_{L}}(q)$ is defined when
$q$ is in $\{q_{1},\ldots, q_{s_{L}}\}$, and
$\Rep_{e_{R}}(q)$ is defined when
$q$ is in $\{q_{1},\ldots, q_{s_{R}}\}$.
The function $f$ will have domain
$\{q_{1},\ldots, q_{s_{L}}\}\times \{q_{1},\ldots, q_{s_{R}}\}$.
For each of the $O(K^{2})$ pairs, $(q_{j},q_{k})$, with $q_{j}\in \{q_{1},\ldots, q_{s_{L}}\}$ and
$q_{k}\in \{q_{1},\ldots, q_{s_{R}}\}$, we perform the following steps.
Let $X_{j}$ stand for $\Rep_{e_{L}}(q_{j})$ and $X_{k}$ stand for $\Rep_{e_{R}}(q_{k})$.
In time bounded by $O(K^{2}n^{c})$, we check whether
$X_{j}\cup X_{k}$ is equivalent under $\approx_{U_{e}}$ to
any of the (at most $K$) representative subsets of $U_{e}$ that
we have already constructed.
If not, then we define $\Rep_{e}(q_{l})$ to be $X_{j}\cup X_{k}$,
where $q_{l}$ is the first label in $\{q_{1},\ldots, q_{K}\}$ that has not already
been assigned to a representative subset of $U_{e}$.
In this case, we set $f(q_{j},q_{k})$ to be $q_{l}$.
However, if $X_{j}\cup X_{k}$ is equivalent under $\approx_{U_{e}}$
to a previously chosen representative, say $\Rep_{e}(q_{m})$, then we set
$f(q_{j},q_{k})$ to be $q_{m}$.
Note that the number of edges in $T_{M}$ is $2n-2$, so this entire procedure takes
$O(n(K^{4}n^{c}))$ steps.

Finally, let the children of the root, $t$, be $u_{L}$ and $u_{R}$, and assume that
$t$ is joined to these children by $e_{L}$ and $e_{R}$.
Assume $\Rep_{e_{L}}(q)$ is defined when $q$ is in $\{q_{1},\ldots, q_{s_{L}}\}$, and
$\Rep_{e_{R}}(q)$ is defined when $q$ is in $\{q_{1},\ldots, q_{s_{R}}\}$.
Again, $f$ has domain $\{q_{1},\ldots, q_{s_{L}}\}\times \{q_{1},\ldots, q_{s_{R}}\}$.
We define $f(q_{j},q_{k})$ to be \indep\ if $\Rep_{e_{L}}(q_{j}) \cup \Rep_{e_{R}}(q_{k})$
is independent, and we let $f(q_{j},q_{k})$ be \dep\ otherwise.
Constructing this function takes $O(K^{2}p(n))$ steps.
Now we have completed the construction of the parse tree $(T_{M},\sigma_{M})$,
and we have done so in $O((8^{\lambda}n^{3.5}+K^{2})p(n)+K^{4}n^{c+1})$ steps.

To complete the proof, we must check that $(T_{M},\sigma_{M})$ genuinely behaves
as a parse tree should.
The automaton $A$ is exactly as in \Cref{shader}.
But the statement of \Cref{pocket} still holds in this case, and can be proved
by the same argument.
There is one point which deserves some attention:
with the notation as in the proof of \Cref{pocket}, the fact that the state $q$ is applied
to $u$ means that $(X_{j}\cup X_{k})\approx_{U_{e}} \Rep_{e}(q)$.
But the definition of $\approx_{U_{e}}$ then implies
$(X_{j}\cup X_{k})\sim_{U_{e}} \Rep_{e}(q)$, and hence
$(Y_{i}\cap U_{e})\sim_{U_{e}} \Rep_{e}(q)$, exactly as in \Cref{pocket}.
The rest of the proof follows exactly as in \Cref{shader}.
\end{proof}

Now \Cref{lounge} follows immediately from \Cref{ermine} and \Cref{jibber}.

\subsection{Automata and $2$-sums}

In \cite{FMN-II}, we extend Hlin\v{e}n\'{y}'s Theorem
to the classes of bicircular matroids and \hgg\ matroids
(where $H$ is a finite group).
If we knew that these classes were efficiently pigeonhole, then this would follow
immediately from \Cref{lounge}, but we only know that the classes of
$3$\dash connected \hgg\ (or bicircular) matroids are efficiently pigeonhole.
In this section, we prove that this is sufficient to extend Hlin\v{e}n\'{y}'s Theorem
to the entire classes (not only the $3$\dash connected members).
Because our arguments here do not depend on the nature of
bicircular or \hgg\ matroids, we operate at a greater level
of generality.

Let $M_{1}$ and $M_{2}$ be matroids on the ground sets
$E_{1}$ and $E_{2}$.
Assume that $E_{1}\cap E_{2}=\{e\}$, where $e$ is neither a loop
nor a coloop in $M_{1}$ or in $M_{2}$.
The \emph{parallel connection}, $P(M_{1},M_{2})$, 
along the \emph{basepoint} $e$, has $E_{1}\cup E_{2}$ as its
ground set.
Let $\mcal{C}_{i}$ be the family of circuits of $M_{i}$ for $i=1,2$.
The family of circuits of $P(M_{1},M_{2})$ is
\[
\mcal{C}_{1}\cup \mcal{C}_{2}\cup \{ (C_{1}-e)\cup (C_{2}-e)\colon
C_{1}\in\mcal{C}_{1},\ C_{2}\in\mcal{C}_{2},\ e\in C_{1}\cap C_{2}\}.
\]
Note that $P(M_{1},M_{2})|E_{i}=M_{i}$, for $i=1,2$.
The \emph{$2$\dash sum} of $M_{1}$ and $M_{2}$, written $M_{1}\oplus_{2} M_{2}$,
is defined to be $P(M_{1},M_{2})\ba e$.

Let $T$ be a tree, where each node, $x$, is labelled with a matroid, $M_{x}$.
Let the edges of $T$ be labelled with distinct elements, $e_{1},\ldots, e_{m}$.
Let $x$ and $y$ be distinct nodes.
We insist that if $x$ and $y$ are not adjacent, then
$E(M_{x})$ and $E(M_{y})$ are disjoint.
If $x$ and $y$ are joined by the edge $e_{i}$, then
$E(M_{x})\cap E(M_{y})=\{e_{i}\}$, where $e_{i}$ is neither a loop nor a
coloop in $M_{x}$ or $M_{y}$.
Such a tree describes a matroid, as we now show.
If $e_{i}$ is an edge joining $x$ to $y$, then contract $e_{i}$ from $T$,
and label the resulting identified node with $P(M_{x}, M_{y})$.
Repeat this procedure until there is only one node remaining.
We use $P(T)$ to denote the matroid labelling this one node.
It is an easy exercise to see that $P(T)$ is well-defined, so that it
does not depend on the order in which we contract the edges of $T$.
We define $\oplus_{2}(T)$ to be $P(T)\ba \{e_{1},\ldots, e_{m}\}$.
If $M$ is a connected matroid, there exists a (not necessarily unique)
tree $T$ satisfying $M=\oplus_{2}(T)$ where every node of the tree is
labelled with a $3$\dash connected matroid.

\begin{definition}
\label{toffee}
Let $\Delta$ be a succinct representation of \mcal{M}, a
class of matroids.
We say that $\Delta$ is \emph{minor-compatible}
if there is a polynomial-time algorithm which will accept
any tuple $(\Delta(M),X,Y)$ when $M\in\mcal{M}$ and $X$ and
$Y$ are disjoint subsets of $E(M)$, and return
a string of the form $\Delta(M/X\ba Y)$.
\end{definition}

It is clear that representating graphic matroids with graphs
or representable matroids with matrices gives us examples of
minor-compatible succinct representations.

\begin{theorem}
\label{kibble}
Let \mcal{M} be a minor-closed class of matroids with a
minor-compatible representation, $\Delta$.
Assume that $\{M\in \mcal{M}\colon M\ \text{is $3$-connected}\}$
is efficiently pigeonhole.
There is a fixed-parameter tractable algorithm (with respect to the parameter
of branch-width) which accepts as input any $\Delta(M)$ when $M\in \mcal{M}$
and returns a parse tree for $M$.
\end{theorem}

\begin{remark}
\label{cyborg}
\Cref{jibber,kibble} are independent of each other, as we now discuss.
Since any subclass of an efficiently pigeonhole class is
efficiently pigeonhole, we can easily construct an
efficiently pigeonhole class that is not minor-closed, and this
class will therefore not be covered by \Cref{kibble}.
On the other hand, we can construct a minor-closed class \mcal{M}
such that $\{M\in \mcal{M}\colon M\ \text{is $3$-connected}\}$ is
efficiently pigeonhole, and yet \mcal{M} is not even strongly
pigeonhole.
Such a class will be covered by \Cref{kibble}, but not by \Cref{jibber}.
For an example, let $n\geq 3$ be an integer, and let
$U_{2,n}^{+}$ be the rank\dash $2$ matroid obtained from
$U_{2,n}$ by replacing each element with a parallel pair.
If $U$ contains exactly one element from each parallel pair,
then it is $3$\dash separating, and yet $\sim_{U}$ has
at least $n$ equivalence classes.
So if \mcal{M} is the smallest minor-closed class containing
$\{U_{2,n}^{+}\colon n\geq 3\}$, then \mcal{M} is not strongly pigeonhole.
However, every $3$\dash connected member of \mcal{M}
is uniform.
It is therefore not difficult to show that
$\{M\in \mcal{M}\colon M\ \text{is $3$-connected}\}$
is efficiently pigeonhole with respect to any minor-compatible
representation.
(See the proof of \cite[Proposition 3.5]{FMN-II}.)
\end{remark}

\begin{proof}[Proof of \textup{\Cref{kibble}}.]
The ideas here are very similar to those in the proof of \Cref{jibber}, but
there are several technical complications introduced by the fact that
we have to deal with non-$3$\dash connected matroids
as a separate case.
Let $M\in \mcal{M}$ be a matroid with ground set $E$ and
branch-width $\lambda$.
We assume that we are given the description $\Delta(M)$.
The algorithm we describe in this proof runs in polynomial-time,
and to ensure that it is fixed-parameter tractable with respect to $\lambda$,
we will be careful to observe that whenever we call upon a polynomial-time
subroutine, $\lambda$ does not appear in the exponent of the running time.

To start, we consider the case that $M$ is connected, and at the end
of the proof we will show that this is sufficient to establish the entire
\namecref{kibble}.

Discussion in \cite{BC95} shows that we can use a `shifting' algorithm
to find a $2$-separation of $M$, if it exists.
This takes $O(|E|^{3})$ oracle calls.
Therefore we can test whether $M$ is $3$-connected in polynomial time.
If $M$ is $3$\dash connected, then we use \Cref{jibber} to construct a parse tree for $M$.
So henceforth we assume that $M$ is connected but not $3$\dash connected.

\textbf{Constructing $T_{M}$.}
We have noted that it takes $O(|E|^{3})$ oracle calls to find a $2$-separation of $M$.
Assume that $(U_{1},U_{2})$ is such a separation.
Then $M$ can be expressed as the $2$-sum of matroids $M_{1}$ and
$M_{2}$, where the ground set of $M_{i}$ is $U_{i}\cup e$, and $e$ is an element
not in $E$.
Both $M_{1}$ and $M_{2}$ are isomorphic to minors of $M$, and hence are in \mcal{M}.
If $B_{1}$ is a basis of $M|U_{1}$, and $B$ is a basis of $M$ containing
$B_{1}$, then $B\cap U_{2}$ does not span $U_{2}$, so we can let $x$ be an element
in $U_{2}-\cl_{M}(B\cap U_{2})$.
Now $M_{1}$ can be produced from $M$ by contracting $B\cap U_{2}$ and deleting
all elements in $U_{2}-(B\cup x)$.
We then relabel $x$ as $e$.
From this discussion, and the fact that $\Delta$ is minor-compatible,
it follows that we can construct $\Delta(M_{1})$ and $\Delta(M_{2})$ in
polynomial time.
By reiterating this procedure, we can construct a labelled tree, $T'$, such that
$M=\oplus_{2}(T')$.
Each node, $x$, of $T'$ is labelled by a $3$\dash connected matroid, $M_{x}$,
with at least three elements, and for each such node we have an
associated string $\Delta(M_{x})$.
Let the edge labels of $T'$ be $e_{1},\ldots, e_{m}$.
We arbitrarily choose to subdivide $e_{m}$ to make a root vertex.
Say that $e_{m}$ joins $x_{L}$ to $x_{R}$ in $T'$.
We delete $e_{m}$, add a new node, $t$, and edges
$e_{m,L}$ and $e_{m,R}$ joining $t$ to $x_{L}$ and $x_{R}$.
At the same time, we relabel $e_{m}$ as $e_{m,L}$ in $M_{x_{L}}$ and
as $e_{m,R}$ in $M_{x_{R}}$.
Let $T$ be the tree that we obtain in this way.
We think of $t$ as being the \emph{root} of $T$.
We associate $t$ with the matroid $M_{t}$, which is a copy of
$U_{1,2}$ with ground set $\{e_{m,L},e_{m,R}\}$.
Note that $\oplus_{2}(T)=\oplus_{2}(T')=M$.

For each non-root node, $x$, of $T$, the labelling matroid
$M_{x}$ is isomorphic to a minor of $M$.
Therefore $\bw(M_{x}) \leq \lambda$ \cite[Proposition 14.2.3]{Oxl11}.
We use \Cref{lychee} to construct a branch-decomposition of $M_{x}$
with width at most $3\lambda+1$.
Let $T_{x}$ be the tree underlying the branch-decomposition of $M_{x}$,
and let $\varphi_{x}$ be the bijection from $E(M_{x})$ to the leaves of $T_{x}$.
We define the tree $T_{t}$ to be a path of two edges, and we define
$\varphi_{t}$ so that it applies the labels $e_{m,L}$ and $e_{m,R}$ to the leaves
and $t$ to the middle vertex.
We say that $\varphi_{t}(e_{m,L})$ is the \emph{left child}
of $t$ and $\varphi_{t}(e_{m,R})$ is the \emph{right child}.

Let $x$ be a non-root node in $T$ and
consider the path in $T$ from $x$ to $t$.
Let $e_{\alpha}$ be the first edge in this path, so that $e_{\alpha}$ is a basepoint in the
ground set of $M_{x}$.
Then we say that $e_{\alpha}$ is the \emph{parent basepoint} of $T_{x}$.
For each internal vertex, $u$, of $T_{x}$, note that $u$ is adjacent to
two vertices that are not in the path from $u$ to $\varphi_{x}(e_{\alpha})$,
where $e_{\alpha}$ is the parent basepoint of $T_{x}$.
We say that these two vertices are the \emph{children} of $u$,
and we make an arbitrary distinction between the
\emph{left child} and the \emph{right child}.

The collection $\cup \{T_{x}\}$, where $x$ ranges over all nodes
of $T$, forms a forest that we now assemble into a single tree, $T_{M}$.
For each edge, $e_{\alpha}$, in $\{e_{1},\ldots, e_{m-1},e_{m,L}, e_{m,R}\}$,
we perform the following operation.
Let the node $x$ of $T$ be chosen so that
$e_{\alpha}$ is the parent basepoint of $T_{x}$,
and let $y$ be the other end-vertex of $e_{\alpha}$ in $T$.
Let $u$ be the vertex of $T_{x}$ that is adjacent to the leaf $\varphi_{x}(e_{\alpha})$.
We delete $\varphi_{x}(e_{\alpha})$ from $T_{x}$ and then identify $u$ with the
leaf $\varphi_{y}(e_{\alpha})$ in $T_{y}$.
We say that the edge in $T_{y}$ that is now incident with $u$ is a
\emph{basepoint} edge in $T_{M}$.
If $u$ is a non-leaf vertex of $T_{x}$, we allow $u$ to carry its
children over from $T_{x}$ to $T_{M}$.
In the case that a child of $u$ in $T_{x}$ represents a basepoint element,
$e_{\alpha}$, then that child of $u$ in $T_{M}$ will be an internal vertex
of another tree, $T_{y}$.
Now $T_{M}$ is a rooted tree where every non-leaf vertex has a left child and a
right child.
\Cref{fig6} illustrates this construction by showing the tree $T'$, along with the
collection of decompositions $\cup \{T_{x}\}$.
In these diagrams, the basepoints of the parallel connections are coded via
colour.
In \Cref{fig7}, we have assembled these trees together into the tree $T_{M}$.

\begin{figure}[htb]
\centering
\includegraphics[scale=1.1]{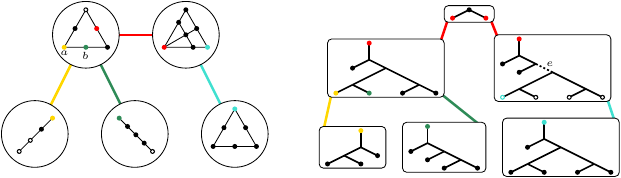}
\caption{The decomposition tree, $T'$, and the decompositions $\cup \{T_{x}\}$.}
\label{fig6}
\end{figure}

\begin{figure}[htb]
\centering
\includegraphics[scale=1.1]{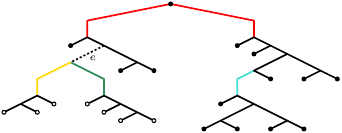}
\caption{The tree $T_{M}$.}
\label{fig7}
\end{figure}

Every edge of $T_{M}$ is an edge of exactly one tree $T_{x}$, where $x$ is a node of $T$.
Our method of construction means that if $u$ is a non-leaf vertex of $T_{M}$, then both
the edges joining $u$ to its children are edges of the same tree $T_{x}$.
Moreover, if $x$ is a non-root node of $T$, then the only edge of $T_{x}$ not contained
in $T_{M}$ is the one incident with $\varphi_{x}(e_{\alpha})$, where $e_{\alpha}$ is the
parent basepoint of $T_{x}$.
Let $e$ be any edge of $T_{M}$, and let $u$ be the end-vertex
of $e$ that is further away from $t$ in $T_{M}$.
Then we say that $u$ is the \emph{bottom} vertex of $e$.

Note that there is a bijection, $\varphi_{M}$, from $E$ to the
leaves of $T_{M}$.
In particular, $\varphi_{M}$ restricted to $E(M_{x})\cap E$
is equal to $\varphi_{x}$ restricted to the same set, for any node $x$.
It is easy to check that if the set $U$ is displayed by an edge of $T_{M}$, then
$\lambda_{M}(U)\leq 3\lambda$.
This is obvious when $U$ is a subset of $E(M_{x})$, where $x$ is a non-root
node of $T$, for then $U$ is also displayed by the tree $T_{x}$.
It is only a little more difficult to verify when $U$ is not contained in
$E(M_{x})$ for any $x$.

\textbf{Defining three pieces of notation.}
Next we describe three related notations for subsets of $E$ and $E(M_{x})$.
Let $e$ be any edge of $T_{M}$, and let $\desc(e)$ be the set of
elements $z\in E$ such that the path in $T_{M}$
from $\varphi_{M}(z)$ to $t$ passes through $e$.
In \Cref{fig7}, when $e$ is the dashed edge, the set $\desc(e)$ is indicated
by the hollow vertices.
Note that $\desc(e)$ is not necessarily contained in $E(M_{x})$ for any node
$x$ of $T$, but that it is contained in $E$.

Next, we let $e$ be any edge of $T_{M}$ and let $x$ be the node of $T$
such that $e$ is an edge of the tree $T_{x}$.
If $x$ is a non-root node of $T$, then we can
assume that $e_{\alpha}$ is the parent basepoint of $T_{x}$.
We let $U_{e}$ be the set of elements $z\in E(M_{x})$ such
that the path in $T_{x}$ from $\varphi_{x}(z)$ to $\varphi_{x}(e_{\alpha})$ contains $e$.
If $x$ is the root $t$, and $e$ joins $t$ to $\varphi_{t}(e_{m,L})$ then
we define $U_{e}$ to be $\{e_{m,L}\}$, and if $e$ joins $t$ to
$\varphi_{t}(e_{m,R})$, then we define $U_{e}$ to be $\{e_{m,R}\}$.
In \Cref{fig6}, the righthand diagram contains a dashed edge $e$, and the
set $U_{e}$ is indicated by hollow vertices.
Note that $U_{e}$ is a subset of $E(M_{x})$, and unlike $\desc(e)$, the set $U_{e}$
may not be contained in $E$, as it may contain basepoint elements.
As $U_{e}$ is displayed by the edge $e$ in $T_{x}$, we have that
$\lambda_{M_{x}}(U_{e})\leq 3\lambda$.

Finally, we let $Y_{i}$ be any subset of $E$, and we let $x$ be a node of $T$.
We recursively describe a subset, $\closure{Y_{i}}_{x}\subseteq E(M_{x})$.
First, assume that $x$ is a leaf of $T$.
Then $\closure{Y_{i}}_{x}$ is simply $Y_{i}\cap E(M_{x})$.
Now we assume that $x$ is not a leaf.
Let $e_{\alpha_{1}},\ldots, e_{\alpha_{s}}$ be the labels
of edges in $T$ that are incident with $x$
but not on the path from $x$ to $t$.
Now define $\closure{Y_{i}}_{x}$ so that it contains
$Y_{i} \cap E(M_{x})$, along with any basepoint
$e_{\alpha_{j}}$ such that if
$y$ labels the other node incident with $e_{\alpha_{j}}$, then
$e_{\alpha_{j}}\in \cl_{M_{y}}(\closure{Y_{i}}_{y})$.
Note that this means that any element of $\closure{Y_{i}}_{x}$ is
either contained in $Y_{i}\cap E(M_{x})$, or is a basepoint element.
In any case, every element of $\closure{Y_{i}}_{x}$ is in $E(M_{x})$.
In \Cref{fig6}, we let $x$ be the top-left node in the lefthand
diagram, and we let $Y_{i}$ be the set indicated by the hollow vertices.
Then $\closure{Y_{i}}_{x}$ contains the single hollow vertex in $E(M_{x})$,
as well as the element $a$, but \emph{not} the element $b$.
Note that by construction, 
$(Y_{i}\cap E(M_{x}))\subseteq \closure{Y_{i}}_{x}\subseteq E(M_{x})$.

With these definitions established, we can proceed.

\textbf{Constructing representative subsets.}
Let $e$ be any edge of $T_{M}$, and let $U_{e}\subseteq E(M_{x})$ be as defined above.
Recall that $\lambda_{M_{x}}(U_{e})\leq 3\lambda$.
Noting that $\{M\in \mcal{M}\colon M\ \text{is $3$-connected}\}$ is
efficiently pigeonhole and $M_{x}$ is $3$-connected, we refer to
\Cref{yakuza}, and we let $\pi$ be the function from that \namecref{yakuza}.
Let $K$ be $\pi(3\lambda)$, and note that $K$ is constant with respect to the size of $E$.
Let $\approx_{U_{e}}$ be the equivalence relation from \Cref{yakuza}.
Then we can decide whether two subsets of $U_{e}$
are equivalent under $\approx_{U_{e}}$ in time bounded by
$O(K|E|^{c})$.
Furthermore, $\approx_{U_{e}}$ has at most $K$ equivalence classes.

Note that $T_{M}$ has exactly $2|E|-2$ edges.
For each such edge, $e$, we will construct, in polynomial time, a set
of representatives such that each representative is a subset of $U_{e}$.
Each representative will be an independent subset of $U_{e}$, and distinct
representatives will represent different $(\approx_{U_{e}})$\dash classes.
We will apply the labels $q_{1},\ldots, q_{K}$ to representative
subsets, and use $\Rep_{e}(q)$ to denote the representative with
label $q$, assuming that it exists.
We do not claim that our set of representatives is complete,
so there may be $(\approx_{U_{e}})$\dash classes that do not have a
representative.

Let $x$ be the node of $T$ such that $e$ is an edge of $T_{x}$.
If $x\ne t$, then $T_{x}$ has a parent basepoint, $e_{\alpha}$, and
we let $u$ be the end-vertex of $e$ that is further from
$\varphi_{x}(e_{\alpha})$ in $T_{x}$.
If $x=t$, then let $u$ be the end-vertex of $e$ that is not the root.

First assume that $u$ is a leaf in $T_{x}$.
Then $U_{e}=\{\varphi_{x}^{-1}(u)\}$.
In this case, we choose $\emptyset$ as a representative,
and apply the label $q_{1}$ to it, so that $\emptyset=\Rep_{e}(q_{1})$.
Because \mcal{M} has a succinct representation, we can check in
polynomial time whether $U_{e}$ is dependent.
If so, we take no further action, so assume that $U_{e}$ is independent in $M_{x}$.
In polynomial time we can check whether
$U_{e}\approx_{U_{e}} \emptyset$ holds.
If so, then we are done.
If $U_{e}\not\approx_{U_{e}} \emptyset$, then
we choose $U_{e}$ as the representative with label $q_{2}$, so that
$U_{e}=\Rep_{e}(q_{2})$.

Now we assume that $u$ is not a leaf of $T_{x}$.
Let $e_{L}$ and $e_{R}$ be the edges joining $u$ to its
children in $T_{x}$.
Recursively, we assume that we have chosen a representative
subset $\Rep_{e_{L}}(q)\subseteq U_{e_{L}}$ whenever
$q$ is in $\{q_{1},\ldots, q_{s_{L}}\}$, and that
$\Rep_{e_{R}}(q)$ is defined when $q$ is in $\{q_{1},\ldots, q_{s_{R}}\}$.
For each pair $(q_{j},q_{k})$ in
$\{q_{1},\ldots, q_{s_{L}}\}\times \{q_{1},\ldots, q_{s_{R}}\}$,
we let $X_{j}$ stand for $\Rep_{e_{L}}(q_{j})$ and
$X_{k}$ stand for $\Rep_{e_{R}}(q_{k})$.
If $X_{j}\cup X_{k}$ is dependent in $M_{x}$, then we
move to the next pair.
Assuming that $X_{j}\cup X_{k}$ is independent,
we check in polynomial time whether
$X_{j}\cup X_{k}$ is equivalent under $\approx_{U_{e}}$ to
any of the representative subsets of $U_{e}$ that we have
already constructed.
If so, we are done.
If not, then we let $q_{l}$ be the first label in
$\{q_{1},\ldots, q_{K}\}$ not already
assigned to a subset of $U_{e}$, and we define
$\Rep_{e}(q_{l})$ to be $X_{j}\cup X_{k}$.

\textbf{Labelling the vertices.}
The alphabet of our automaton is going to contain a set of functions.
Our next job is show how we apply, in polynomial time,
functions in the alphabet to the vertices of $T_{M}$.
Let $u$ be a vertex of $T_{M}$, and assume that $u$ is the bottom
vertex of the edge $e$.
Let $x$ be the node of $T$ such that $e$ is an edge in the tree $T_{x}$.

First, we assume that $u$ is a leaf of $T_{M}$.
Then $U_{e}=\varphi_{M}^{-1}(u)$.
We label $u$ with a function, $f$, with the domain $\{0,1\}$.
Set $f(0)$ to be $q_{1}$, recalling that
$\emptyset$ is the representative $\Rep_{e}(q_{1})$.
If $U_{e}$ is dependent in $M_{x}$, then we set
$f(1)$ to be the symbol \dep.
Assume that $U_{e}$ is independent.
If $U_{e}\approx_{U_{e}} \emptyset$, then we set $f(1)$ to be $q_{1}$.
Otherwise we set $f(1)$ to be $q_{2}$, recalling that in this case
$U_{e}=\Rep_{e}(q_{2})$.
Henceforth we assume that $u$ is not a leaf of $T_{M}$.
Let $e_{L}$ and $e_{R}$ be the edges joining $u$ to its
children in $T_{M}$.

Assume that $e$ is not a basepoint edge.
This implies that $U_{e}$ is the disjoint union of $U_{e_{L}}$
and $U_{e_{R}}$.
(If $e$ were a basepoint edge, then $U_{e}$ would be a singleton
subset of $E(M_{x})$, whereas
$U_{e_{L}}$ and $U_{e_{R}}$ would be
subsets of $E(M_{y})$ for some other node $y$ of $T$.)
Assume that $\Rep_{e_{L}}(q)$ is defined when $q$ is
in $\{q_{1},\ldots, q_{s_{L}}\}$, and that
$\Rep_{e_{R}}(q)$ is defined when $q$ is in
$\{q_{1},\ldots, q_{s_{R}}\}$.
In this case, we label $u$ with a function, $f$, having
$\{\dep,q_{1},\ldots, q_{s_{L}}\}\times \{\dep,q_{1},\ldots, q_{s_{R}}\}$
as its domain.
We define the output of $f$ to be \dep\ on any input that includes
the symbol \dep.
Now assume that $X_{j}$ is $\Rep_{e_{L}}(q_{j})$ and $X_{k}$ is
$\Rep_{e_{R}}(q_{k})$, for some $1\leq j\leq s_{L}$ and some
$1\leq k\leq s_{R}$.
If $X_{j}\cup X_{k}$ is dependent in $M_{x}$, then define
$f(q_{j},q_{k})$ to be \dep.
Otherwise, $X_{j}\cup X_{k}$ is equivalent under $\approx_{e}$ to
some representative subset of $U_{e}$.
We find this representative, say $\Rep_{e}(q_{l})$, in polynomial
time, and we set $f(q_{j},q_{k})$ to be $q_{l}$.

Now we assume that $e$ is a basepoint edge.
Assume that in $T_{x}$, $e$ is incident with the leaf $\varphi_{x}(e_{\alpha})$,
where $e_{\alpha}$ is in $\{e_{1},\ldots, e_{m-1},e_{m,L},e_{m,R}\}$.
Therefore $U_{e}=\{e_{\alpha}\}$.
Note that $e_{L}$ and $e_{R}$ are edges of $T_{y}$, where
$y$ is the node of $T$ joined to $x$ by $e_{\alpha}$.
Assume that $\Rep_{e_{L}}(q)$ has been chosen when
$q\in \{q_{1},\ldots, q_{s_{L}}\}$, and
$\Rep_{e_{R}}(q)$ is defined when
$q\in \{q_{1},\ldots, q_{s_{R}}\}$.
We will apply to $u$ a function, $f$, whose domain is again
$\{\dep,q_{1},\ldots, q_{s_{L}}\}\times \{\dep,q_{1},\ldots, q_{s_{R}}\}$,
and whose codomain is $\{\dep,q_{1},q_{2}\}$.
The output of $f$ is \dep\ on any input including \dep.
Consider the input $(q_{j},q_{k})$.
Let $X_{j}$ and $X_{k}$ be $\Rep_{e_{L}}(q_{j})$ and
$\Rep_{e_{R}}(q_{k})$ respectively.
If $X_{j}\cup X_{k}$ is dependent, then we set
$f(q_{j},q_{k})$ to be \dep.
Now we assume that $X_{j}\cup X_{k}$ is independent.
If $X_{j}\cup X_{k}\cup \{e_{\alpha}\}$ is independent in $M_{y}$ then
we set $f(q_{j},q_{k})$ to be $q_{1}$.
Otherwise, $X_{j}\cup X_{k}\cup \{e_{\alpha}\}$ is dependent in $M_{y}$,
and we set $f(q_{j},q_{k})$ to be
$q_{1}$ if $U_{e}\approx_{U_{e}}\emptyset$ holds, and
$q_{2}$ if $U_{e}\not\approx_{U_{e}}\emptyset$.

Finally, the root $t$ is labelled with a function, $f$,
that takes $\{\dep,q_{1},q_{2}\}^{2}$ as input.
Any ordered pair that contains \dep\ produces \dep\ as output.
Similarly, $f(q_{2},q_{2})=\dep$.
Any other ordered pair produces the symbol \indep\ as output.

Now we have described the function that we apply to each vertex of $T_{M}$.
Let $\sigma_{M}$ be the labelling that applies these functions.
Thus $(T_{M},\sigma_{M})$ is a $\Sigma$\dash tree,
where $\Sigma$ contains functions whose domain is
either $\{0,1\}$ or sets of the form
$\{\dep,q_{1},\ldots, q_{s_{L}}\}\times \{\dep,q_{1},\ldots, q_{s_{R}}\}$,
and whose codomain is
$\{\dep, \indep,q_{1},\ldots, q_{K}\}$.

\textbf{Constructing the automaton.}
Now that we have shown how to efficiently construct the
$\Sigma$\dash tree $(T_{M},\sigma_{M})$, it is time to
consider the workings of the automaton, $A$.
The state space, $Q$, of $A$ is the set
$\{\dep, \indep,q_{1},\ldots, q_{K}\}$.
The alphabet is $\Sigma\cup \Sigma\times \{0,1\}^{\{i\}}$,
where $\Sigma$ is the set of functions into $Q$ that we have
previously described.
The only accepting state is \indep.
The transition rule, $\delta_{0}$, acts as follows.
If $f$ is a function from $\{0,1\}$ into $Q$, and
$s$ is a function in $\{0,1\}^{\{i\}}$, then $\delta_{0}(f,s)=\{f(s(i))\}$.
Similarly, $\delta_{2}$ is defined so that
if $f$ is a function in $\Sigma$, and
$(\alpha,\beta)$ is in the domain of $f$, then
$\delta_{2}(f,\alpha,\beta)=\{f(\alpha,\beta)\}$.
This completes the description of $A$.
Note that it is a deterministic automaton.

\textbf{Proof of correctness.}
We must now prove that $(T_{M},\sigma_{M})$ truly is a
parse tree relative to the automaton $A$.
That is, we must prove that $A$ accepts a subset of the
leaves of $T_{M}$ if and only if the corresponding set
is independent in $M$.

\begin{lemma}
\label{rabbit}
Let $Y_{i}$ be a subset of $E$, and let $u$ be a non-leaf
vertex of $T_{M}$.
Let $e_{L}$ and $e_{R}$ be the edges of $T_{M}$ joining $u$
to its children.
Let $y$ be the node of $T$ such that $e_{L}$ and $e_{R}$ are
edges of $T_{y}$.
\begin{enumerate}[label=\textup{(\roman*)}]
\item If $\closure{Y_{i}}_{y}\cap (U_{e_{L}}\cup U_{e_{R}})$
is dependent in $M_{y}$, then
$Y_{i}\cap (\desc(e_{L})\cup \desc(e_{R}))$ is dependent in $M$.
\item If $Y_{i}\cap \desc(e_{L})$ and $Y_{i}\cap \desc(e_{R})$ are
independent in $M$, but
$Y_{i}\cap (\desc(e_{L})\cup \desc(e_{R}))$ is dependent,
then $\closure{Y_{i}}_{y}\cap (U_{e_{L}}\cup U_{e_{R}})$
is dependent in $M_{y}$.
\end{enumerate}
\end{lemma}

\begin{proof}
We start by defining $D$, a set of nodes in $T$.
Let $y'$ be a node in $T$.
If there exists $d \in \desc(e_{L})\cup \desc(e_{R})$ such that
the path in $T_{M}$ from $\varphi_{M}(d)$ to $u$ uses an
edge in the tree $T_{y'}$, then $y'$ is in $D$, and otherwise $y'\notin D$.
If $y'\ne y$ and $d$ is in $\desc(e_{L})$, we say $y'$ is a
\emph{left} vertex, if $d$ is in $\desc(e_{R})$, then $y'$
is a \emph{right} vertex.
We say that $y$ is both a left and a right vertex of $D$, and note that
any vertex in $D-y$ is either left or right, but not both.
Let $y_{0},\ldots, y_{s}$ be an ordering of the vertices in $D$
such that $y_{0}=y$, and whenever $y_{k}$ is on the path from $y_{j}$
to $y$ in $T$, $k\leq j$.

To prove (i), we let $C$ be a circuit of $M_{y}$ contained in
$\closure{Y_{i}}_{y}\cap (U_{e_{L}}\cup U_{e_{R}})$.
We will construct a sequence of circuits, $C_{0},\ldots, C_{s}$,
of $P(T)$ such that:
\begin{enumerate}[label=\textup{(\alph*)}]
\item $C_{j}$ is contained in
\[(\closure{Y_{i}}_{y}\cap (U_{e_{L}}\cup U_{e_{R}}))\cup \bigcup_{z=1}^{j}\closure{Y_{i}}_{y_{s}}\]
for each $j$, and
\item if $0\leq k\leq j$, and $e'$ is a basepoint edge on the path from $y_{k}$ to $y$
in $T$, then $e'\notin C_{j}$.
\end{enumerate}
Assume we succeed in constructing this sequence.
Then $C_{s}$ does not contain any element in
$\{e_{1},\ldots, e_{m-1},e_{m,L},e_{m,R}\}$, so it is a circuit of
$P(T)\ba \{e_{1},\ldots, e_{m-1},e_{m,L},e_{m,R}\}=M$, and is
contained in $Y_{i}\cap (\desc(e_{L})\cup \desc(e_{R}))$.
So at this point the proof of (i) will be complete.

For $C_{0}$, we can just use $C$.
Assume we have constructed $C_{j-1}$.
Let $e_{\alpha}$ be the parent basepoint of $T_{y_{j}}$, so that
$e_{\alpha}$ is the first edge on the path in $T$ from $y_{j}$ to $y$.
Assume that $e_{\alpha}$ joins $y_{j}$ to $y_{k}$, where $k<j$.
This means that $e_{\alpha}$ is in $E(M_{y_{k}})$.
If $e_{\alpha}\notin C_{j-1}$, then we set $C_{j}$ to be $C_{j-1}$ and we are done.
Therefore we assume that $e_{\alpha}$ is in $C_{j-1}$.
Because $C_{j-1}$ is contained in the union of
\[(\closure{Y_{i}}_{y}\cap (U_{e_{L}}\cup U_{e_{R}})\quad\text{with}\quad \closure{Y_{i}}_{y_{1}}\cup\cdots\cup \closure{Y_{i}}_{y_{j-1}}\]
and $e_{\alpha}$ is in $E(M_{y_{k}})$ it follows that 
$e_{\alpha}$ is in $\closure{Y_{i}}_{y_{k}}$.
The definition of $\closure{Y_{i}}_{y_{k}}$ now means that $e_{\alpha}$ is
in $\cl_{M_{y_{j}}}(\closure{Y_{i}}_{y_{j}})$.
Let $C'$ be a circuit of $M_{y_{j}}$ such that
$e_{\alpha}\in C'\subseteq (\closure{Y_{i}}_{y_{j}}\cup e_{\alpha})$.
The definition of the parallel connection means that
$(C_{j-1}-e_{\alpha})\cup (C'-e_{\alpha})$ is a circuit of $P(T)$, so we
set $C_{j}$ to be this circuit.
This shows that we can construct the claimed sequence of circuits, and
completes the proof of (i).

Now we prove (ii).
Assume that $Y_{i}\cap \desc(e_{L})$ and
$Y_{i}\cap \desc(e_{R})$ are independent in $M$, but that
$C$ is a circuit contained in $Y_{i}\cap (\desc(e_{L})\cup \desc(e_{R}))$.
We construct a sequence $C_{s},C_{s-1},\ldots, C_{0}$ of circuits
of $P(T)$ such that:
\begin{enumerate}[label=\textup{(\alph*)}]
\item $C_{j}$ is contained in
\[(\closure{Y_{i}}_{y}\cap (U_{e_{L}}\cup U_{e_{R}}))\cup\bigcup_{z=1}^{j}\closure{Y_{i}}_{y_{k}}\]
for each $j$, and
\item for each $j$, there is a left vertex $y_{L}$ and a right vertex $y_{R}$
such that $C_{j}$ contains elements of both
$\closure{Y_{i}}_{y_{L}}$ and $\closure{Y_{i}}_{y_{R}}$.
\end{enumerate}
Assuming we succeed in constructing this sequence,
$C_{0}$ will certify that $\closure{Y_{i}}_{y}\cap (U_{e_{L}}\cup U_{e_{R}})$
is dependent.

Note that $C$ is contained in neither
$Y_{i}\cap \desc(e_{L})$ nor
$Y_{i}\cap \desc(e_{R})$.
From this it follows that we can take $C_{s}$ to be $C$.
Now assume that we have constructed $C_{j}$.
If $C_{j}$ contains no elements of $\closure{Y_{i}}_{y_{j}}$, then
we set $C_{j-1}$ to be $C_{j}$.
So assume that $C_{j}\cap \closure{Y_{i}}_{y_{j}}\ne \emptyset$.
Let $e_{\alpha}$ be the parent basepoint of $T_{y_{j}}$,
and assume that $e_{\alpha}$ joins $y_{j}$ to $y_{k}$ in $T$,
where $k<j$.
It cannot be the case that $C_{j}$ is a circuit of
$\closure{Y_{i}}_{y_{j}}$, or else condition (b) would be
violated.
Therefore $C_{j}$ can be expressed as
$(C'-e_{\alpha})\cup (C_{j-1}-e_{\alpha})$,
where $C'$ and $C_{j-1}$ are circuits of $P(T)$
containing $e_{\alpha}$, and
$C'$ is a circuit of $M_{y_{j}}$, while
$C_{j-1}$ intersects $E(M_{y_{j}})$ only in $e_{\alpha}$.
Note that the circuit $C'$ implies that 
$e_{\alpha}$ is in $\closure{Y_{i}}_{y_{k}}$.
If $y_{j}$ is a left vertex, then so is $y_{k}$, so $C_{j-1}$ must
also contain an element from $\closure{Y_{i}}_{y_{R}}$, where $y_{R}$ is
some right vertex.
Therefore $C_{j-1}$ is the desired next circuit in the sequence.
The symmetric argument applies when $y_{j}$ and $y_{k}$ are
both right vertices.
\end{proof}

\begin{lemma}
\label{gelato}
Let $Y_{i}$ be a subset of $E$.
Assume that $u$ is the bottom vertex of the edge $e$ in $T_{M}$.
Let $x$ be the node of $T$ such that $e$ is an edge of $T_{x}$.
Let $q$ be the state applied to $u$ by the run of $A$  on
$\enc(T_{M},\sigma_{M},\varphi_{M},\{Y_{i}\})$.
Then $q=\dep$ if and only if $Y_{i}\cap \desc(e)$
is dependent in $M$.
If $Y_{i}\cap \desc(e)$ is independent, then
$q=q_{l}$ for some value $l$, and
$(\closure{Y_{i}}_{x}\cap U_{e}) \sim_{U_{e}} \Rep_{e}(q_{l})$.
\end{lemma}

\begin{proof}
We assume that the \namecref{gelato} fails for the vertex $u$, and
that subject to this constraint, $u$ has been chosen so that it is as
far away from $t$ as is possible in $T_{M}$.
Let $f$ be the function applied to $u$ by the labelling $\sigma_{M}$.

\begin{claim}
\label{jabber}
$u$ is not a leaf of $T_{M}$.
\end{claim}

\begin{proof}
Let us assume that $u$ is a leaf.
Note that $\desc(e)=U_{e}=\{\varphi_{x}^{-1}(u)\}$.
The label applied to $u$ in the $\Sigma$\dash tree
$\enc(T_{M},\sigma_{M},\varphi_{M},\{Y_{i}\})$ is
$(f,s)$, where $s\in\{0,1\}^{\{i\}}$ is the function such that $s(i)=1$ if
$\varphi_{x}^{-1}(u)$ is in $Y_{i}$, and otherwise
$s(i)=0$.
We have defined $A$ in such a way that
$q=f(s(i))$.

Assume that $Y_{i}\cap \desc(e)$ is dependent.
The only way this can occur is if $\varphi_{x}^{-1}(u)$ is a loop
contained in $Y_{i}$.
In this case $q=f(s(i))=f(1)$, and $f(1)=\dep$, by the construction of $f$.
Therefore $u$ does not provide a counterexample to the \namecref{gelato}, contrary
to assumption.
Hence $Y_{i}\cap \desc(e)$ is independent in $M$.

Next assume that $q=\dep$.
But $q$ is $f(s(i))$, and this takes the value $\dep$
only if $s(i)=1$ and $Y_{i}\cap \desc(e)=U_{e}$, and furthermore,
this set is dependent.
Again, $u$ does not provide a counterexample, so we conclude that $q\ne \dep$.

Observe that $e$ is not a basepoint edge,
as this would imply $|\desc(e)|\geq 2$, and this is not the
case when $u$ is a leaf.
From this we deduce that
$\closure{Y_{i}}_{x}\cap U_{e}=Y_{i}\cap U_{e}$.
Assume that $Y_{i}\cap U_{e}=\emptyset$.
Then $q=f(s(i))=f(0)=q_{1}$, where
$\Rep_{e}(q_{1})$ is the empty set.
Thus $\closure{Y_{i}}_{x}\cap U_{e}$ and $ \Rep_{e}(q)$ are actually
equal, and thus certainly equivalent under $\sim_{U_{e}}$, as desired.
Now we assume that $U_{e}\subseteq Y_{i}$,
so $\closure{Y_{i}}_{x}\cap U_{e}=U_{e}=\{\varphi_{x}^{-1}(u)\}$.
Then $q=f(s(i)) = f(1)$, and this value is either $q_{1}$ or $q_{2}$.
In the former case, $U_{e}\approx_{U_{e}}\emptyset$, so
$(\closure{Y_{i}}_{x}\cap U_{e}) \approx_{U_{e}} \emptyset$.
As $\emptyset$ is $\Rep_{e}(q_{1})$, we are done.
Therefore we consider the case that $q=q_{2}$.
In this case $\Rep_{e}(q_{2})=U_{e}=\closure{Y_{i}}_{x}\cap U_{e}$,
so \Cref{gelato} holds.
This contradiction means that \Cref{jabber} is proved.
\end{proof}

Because \Cref{jabber} tells us that $u$ is not a leaf, we let $u_{L}$ and $u_{R}$ be
the children of $u$ in $T_{M}$, and we assume that these are the bottom
vertices of the edges $e_{L}$ and $e_{R}$.
Note that $\desc(e)$ is the disjoint union of $\desc(e_{L})$ and
$\desc(e_{R})$.
Observe also that $e_{L}$ and $e_{R}$ are edges of the same tree,
$T_{y}$, where $y$ is a node of $T$ that may or may not be equal to $x$.
If $y\ne x$, then $e$ is a basepoint edge.
Let $q_{L}$ and $q_{R}$ be the states applied to $u_{L}$ and $u_{R}$
by the run of $A$ on $\enc(T_{M},\sigma_{M},\varphi_{M},\{Y_{i}\})$.
Our inductive assumption on $u$ means that \Cref{gelato} holds for
$u_{L}$ and $u_{R}$.

\begin{claim}
\label{bulgar}
$Y_{i}\cap \desc(e_{L})$ and $Y_{i}\cap \desc(e_{R})$ are independent
in $M$, and $\closure{Y_{i}}_{y}\cap U_{e_{L}}$ and
$\closure{Y_{i}}_{y}\cap U_{e_{R}}$ are independent in $M_{y}$.
\end{claim}

\begin{proof}
If $Y_{i}\cap \desc(e_{L})$ is dependent, then so is
$Y_{i}\cap \desc(e)$.
In this case the inductive assumption tells us that $q_{L}=\dep$.
Now the construction of $f$ and $A$ means that $q=\dep$.
But this means that $u$ does not provide us with a counterexample.
Hence $Y_{i}\cap \desc(e_{L})$, and symmetrically
$Y_{i}\cap \desc(e_{R})$, is independent in $M$.

Assume that $\closure{Y_{i}}_{y}\cap U_{e_{L}}$ is dependent in $M_{y}$.
If $u_{L}$ is not a leaf of $T_{M}$ and $e_{L}$ is not a basepoint edge,
then we can apply \Cref{rabbit}~(i) to the two edges connecting $u_{L}$
to its children.
This then implies that $\closure{Y_{i}}_{y}\cap \desc(e_{L})$ is
dependent in $M$, contradicting the conclusion of the previous paragraph.
Therefore $u_{L}$ is a leaf or $e_{L}$ is a basepoint edge.
In either case, $U_{e_{L}}$ is a singleton set, and this set must
contain a loop, as $\closure{Y_{i}}_{y}\cap U_{e_{L}}$ is dependent.
A basepoint cannot be a loop, so $u_{L}$ is a leaf of $T_{M}$.
Thus $U_{e_{L}}=\desc(e_{L})$.
Now the dependence of $\closure{Y_{i}}_{y}\cap U_{e_{L}}$ implies the
dependence of $Y_{i}\cap \desc(e_{L})$, a contradiction.
The \namecref{bulgar} follows by a symmetrical argument for
$\closure{Y_{i}}_{y}\cap U_{e_{R}}$.
\end{proof}

\Cref{bulgar} and the inductive assumption now mean that
$q_{L}=q_{j}$ and $q_{R}=q_{k}$, for some values of $j$ and $k$.
Let $X_{j}$ and $X_{k}$  stand for
$\Rep_{e_{L}}(q_{j})$ and $\Rep_{e_{R}}(q_{k})$.
Then
$(\closure{Y_{i}}_{y} \cap U_{e_{L}}) \sim_{U_{e_{L}}} X_{j}$ and
$(\closure{Y_{i}}_{y} \cap U_{e_{R}}) \sim_{U_{e_{R}}} X_{k}$.

\begin{claim}
\label{wobble}
$\closure{Y_{i}}_{y}\cap (U_{e_{L}}\cup U_{e_{R}})$ is independent in $M_{y}$,
$Y_{i}\cap \desc(e)$ is independent in $M$, and
$\closure{Y_{i}}_{x}\cap U_{e}$ is independent in $M_{x}$.
\end{claim}

\begin{proof}
Assume that
\[\closure{Y_{i}}_{y}\cap (U_{e_{L}}\cup U_{e_{R}})
=(\closure{Y_{i}}_{y} \cap U_{e_{L}})\cup (\closure{Y_{i}}_{y} \cap U_{e_{R}})
\]
is dependent in $M_{y}$.
Then \Cref{icecap} implies that $X_{j}\cup X_{k}$ is dependent in $M_{y}$.
The construction of $f$ and $A$ now means that $q=\dep$.
\Cref{rabbit}~(i) implies that $Y_{i}\cap \desc(e)$ is dependent, so
$u$ fails to provide a counterexample.
Therefore $\closure{Y_{i}}_{y}\cap (U_{e_{L}}\cup U_{e_{R}})$ is independent
in $M_{y}$.

Assume that $Y_{i}\cap \desc(e)$ is dependent in $M$.
As $Y_{i}\cap \desc(e_{L})$ and $Y_{i}\cap \desc(e_{R})$
are independent by \Cref{bulgar}, \Cref{rabbit}~(ii) implies that
$\closure{Y_{i}}_{y}\cap (U_{e_{L}}\cup U_{e_{R}})$ is
dependent in $M_{y}$, contradicting the previous paragraph.

Finally, assume that $\closure{Y_{i}}_{x}\cap U_{e}$ is dependent in $M_{x}$.
Then $x\ne y$, or else $U_{e}$ is the disjoint union of 
$U_{e_{L}}$ and $U_{e_{R}}$, and we have a contradiction to the
first paragraph.
Hence $e$ is a basepoint edge, meaning that $U_{e}$ is a single
element, and this element must be a loop.
A basepoint cannot be a loop, so we have a contradiction.
\end{proof}

Assume that $x=y$, so that $e$, $e_{L}$, and $e_{R}$ are all
edges of $T_{x}$.
In this case $U_{e}$ is the disjoint union of $U_{e_{L}}$ and $U_{e_{R}}$.
\Cref{wobble} says that
$\closure{Y_{i}}_{x}\cap U_{e}=
(\closure{Y_{i}}_{x}\cap U_{e_{L}})\cup (\closure{Y_{i}}_{x}\cap U_{e_{R}})$
is independent in $M_{x}$, so \Cref{icecap} implies that
$X_{j}\cup X_{k}$ is independent.
Therefore $q=q_{l}$ for some $q_{l}$
such that $(X_{j}\cup X_{k})\approx_{U_{e}} \Rep_{e}(q_{l})$.
Hence $(X_{j}\cup X_{k})\sim_{U_{e}} \Rep_{e}(q_{l})$.
We also know from \Cref{icecap} that
$(\closure{Y_{i}}_{x}\cap U_{e})\sim_{U_{e}} (X_{j}\cup X_{k})$.
Therefore $(\closure{Y_{i}}_{x}\cap U_{e})\sim_{U_{e}} \Rep_{e}(q_{l})$
and \Cref{gelato} holds for $u$, a contradiction.

Now we must assume that $y\ne x$, so that $e$ is a basepoint edge.
This means that $e$ is incident with a leaf, $\varphi_{x}(e_{\alpha})$,
in $T_{x}$, and $e_{\alpha}$ is the parent basepoint of $T_{y}$.
Therefore $U_{e}=\{e_{\alpha}\}$, and $e_{\alpha}$ is in
$\closure{Y_{i}}_{x}$ if and only if $e_{\alpha}$ is in
\[\cl_{M_{y}}(\closure{Y_{i}}_{y})=
\cl_{M_{y}}((\closure{Y_{i}}_{y}\cap U_{e_{L}})\cup (\closure{Y_{i}}_{y}\cap U_{e_{R}})).\]
Note that $(\closure{Y_{i}}_{y}\cap U_{e_{L}})\cup (\closure{Y_{i}}_{y}\cap U_{e_{R}})$
is independent in $M_{y}$, by \Cref{wobble}.

Assume that $e_{\alpha}$ is in $\closure{Y_{i}}_{x}$, so that
$\closure{Y_{i}}_{x}\cap U_{e}=U_{e}=\{e_{\alpha}\}$.
In this case
\[
(\closure{Y_{i}}_{y}\cap U_{e_{L}})\cup (\closure{Y_{i}}_{y}\cap U_{e_{R}})\cup\{e_{\alpha}\}
\]
is dependent in $M_{y}$.
From \Cref{icecap}, we have that
$X_{j}\cup X_{k}$ is independent in $M_{y}$, and
equivalent under $\sim_{(U_{e_{L}}\cup U_{e_{R}})}$ to
$(\closure{Y_{i}}_{y}\cap U_{e_{L}})\cup (\closure{Y_{i}}_{y}\cap U_{e_{R}})$.
Therefore $X_{j}\cup X_{k}\cup\{e_{\alpha}\}$ is also dependent in $M_{y}$.
The construction of the function $f$ now means that
$q$ is either $q_{1}$ or $q_{2}$.
In the first case, $U_{e}\approx_{U_{e}}\emptyset$.
Hence $(\closure{Y_{i}}_{x}\cap U_{e})\sim_{U_{e}}\emptyset$,
and as $\emptyset=\Rep_{e}(q_{1})$, 
we see that $u$ satisfies the \namecref{gelato}.
Therefore $q=q_{2}$, and $\Rep_{e}(q_{2})=U_{e}$.
In this case $\closure{Y_{i}}_{x}\cap U_{e}$ and $\Rep_{e}(q_{2})$ are
equal, so $(\closure{Y_{i}}_{x}\cap U_{e})\sim_{U_{e}}\Rep_{e}(q_{2})$ certainly
holds, and we have a contradiction.

Now we must assume that $e_{\alpha}$ is not in
$\closure{Y_{i}}_{x}$.
Hence $\closure{Y_{i}}_{x}\cap U_{e}=\emptyset$.
But in this case
\[
(\closure{Y_{i}}_{y}\cap U_{e_{L}})\cup (\closure{Y_{i}}_{y}\cap U_{e_{R}})\cup\{e_{\alpha}\}
\]
is independent in $M_{y}$.
Using the arguments from the previous paragraph, we show that
$X_{j}\cup X_{k}\cup\{e_{\alpha}\}$ is also independent in $M_{y}$,
so $q=q_{1}$, where $\Rep_{e}(q_{1})=\emptyset$.
Now $\emptyset=\closure{Y_{i}}_{x}\cap U_{e}$, so
$(\closure{Y_{i}}_{x}\cap U_{e})\sim_{U_{e}}\Rep_{e}(q_{1})$ holds,
and this contradiction completes the proof of \Cref{gelato}.
\end{proof}

Now we can prove that $(T_{M},\sigma_{M})$ is a parse tree for $A$.
Let the left child of the root $t$ be $u_{L}$, and let the right child be $u_{R}$.
We let $e_{L}$ and $e_{R}$ be the edges joining $t$ to these children.
Recall that $U_{e_{L}}=\{e_{m,L}\}$ and $U_{e_{R}}=\{e_{m,R}\}$, and
$M_{t}$ is a copy of $U_{1,2}$ on the ground set
$\{e_{m,L},e_{m,R}\}$.
Let $Y_{i}$ be a subset of $E$, and let
$q$ be the state applied to $t$ by the run of $A$ on
$\enc(T_{M},\sigma_{M},\varphi_{M},\{Y_{i}\})$.
We let $q_{L}$ and $q_{R}$ be the states applied to $u_{L}$ and $u_{R}$.

In the first case, we assume that $q=\dep$, and we aim
to show that $Y_{i}$ is dependent.
Our construction of the function labelling $t$ means $(q_{L},q_{R})$ either
contains the symbol $\dep$, or is $(q_{2},q_{2})$.
If $q_{L}=\dep$, then $Y_{i}\cap \desc(e_{m,L})$ is dependent
in $M$ by \Cref{gelato}, and hence $Y_{i}$ is dependent in $M$.
By symmetry, we assume that neither $q_{L}$ nor $q_{R}$ is
\dep, so $(q_{L},q_{R})=(q_{2},q_{2})$.
From this we see that $\Rep_{e_{L}}(q_{2})=U_{e_{L}}=\{e_{m,L}\}$,
and moreover
$(\closure{Y_{i}}_{t}\cap U_{e_{L}})\sim_{U_{e_{L}}} U_{e_{L}}$.
Because $\Rep_{e_{L}}(q_{2})$ is defined,
$U_{e_{L}}\not\approx_{U_{e_{L}}}\emptyset$, so
$\closure{Y_{i}}_{t}\cap U_{e_{L}}$ is not empty.
Therefore $\closure{Y_{i}}_{t}$ contains $e_{m,L}$.
By symmetry, $\closure{Y_{i}}_{t}$ contains $e_{m,R}$.
Now $\closure{Y_{i}}_{t}\cap (U_{e_{L}}\cup U_{e_{R}})
=\{e_{m,L},e_{m,R}\}$ is dependent in $M_{t}$.
\Cref{rabbit}~(i) implies that
$Y_{i}\cap (\desc(e_{L})\cup \desc(e_{R}))=Y_{i}$
is dependent in $M$, exactly as we wanted.

In the second case, we assume that $Y_{i}$ is dependent in $M$.
If $Y_{i}\cap \desc(e_{L})$ is dependent, then
$q_{L}=\dep$ by \Cref{gelato}.
In this case, $q=\dep$, which is what we want.
Therefore we assume by symmetry that
$Y_{i}\cap \desc(e_{L})$ and $Y_{i}\cap \desc(e_{R})$
are independent in $M$.
As $Y_{i}$ is dependent, \Cref{rabbit}~(ii)
implies that
$\closure{Y_{i}}_{t}\cap (U_{e_{L}}\cup U_{e_{R}})
=\closure{Y_{i}}_{t}\cap \{e_{m,L},e_{m,R}\}$ is dependent
in $M_{t}$.
Therefore $\closure{Y_{i}}_{t}= \{e_{m,L},e_{m,R}\}$.
Let $x_{L}$ be the node of $T$ joined to $t$ by $e_{m,L}$, and
define $x_{R}$ similarly.
Then $e_{m,L}\in \closure{Y_{i}}_{t}$ implies that
$e_{m,L}$ is in $\cl_{M_{x_{L}}}(\closure{Y_{i}}_{x_{L}})$.

Since every node of $T$ other than $t$ corresponds to a matroid
with at least three elements, it follows that neither $u_{L}$ nor $u_{R}$
is a leaf in $T_{M}$.
Let $e_{LL}$ and $e_{LR}$ be the edges that join $u_{L}$
to its children: $u_{LL}$ and $u_{LR}$.
Then $Y_{i}\cap\desc(e_{LL})$ and $Y_{i}\cap \desc(e_{LR})$ are
independent in $M$, as they are subsets of $Y_{i}\cap \desc(e_{L})$.
Therefore $A$ applies states $q_{j}$ and $q_{k}$ to
$u_{LL}$ and $u_{LR}$.
Let $X_{j}$ and $X_{k}$ be $\Rep_{e_{LL}}(q_{j})$ and
$\Rep_{e_{LR}}(q_{j})$ respectively.
Then
$(\closure{Y_{i}}_{x_{L}}\cap U_{e_{LL}})\sim_{U_{e_{LL}}} X_{j}$
and
$(\closure{Y_{i}}_{x_{L}}\cap U_{e_{LR}})\sim_{U_{e_{LR}}} X_{k}$
by \Cref{gelato}.
\Cref{icecap} implies that
$\closure{Y_{i}}_{x_{L}}\cap (U_{e_{LL}}\cup U_{e_{LR}})
=\closure{Y_{i}}_{x_{L}}$
is equivalent to $X_{j}\cup X_{k}$ under
$\sim_{(U_{e_{LL}}\cup U_{e_{LR}})}$.
From $e_{m,L}\in \cl_{M_{x_{L}}}(\closure{Y_{i}}_{x_{L}})$, we see that
$\closure{Y_{i}}_{x_{L}}\cup\{e_{m,L}\}$ is dependent in $M_{x_{L}}$.
Therefore $X_{j}\cup X_{k}\cup\{e_{m,L}\}$ is also
dependent.
Because $U_{e_{L}}=\{e_{m,L}\}$ is certainly not equivalent to
$\emptyset$ under $\sim_{U_{e_{L}}}$, we see that
$A$ applies the state $q_{2}$ to $u_{L}$.
By symmetry it applies $q_{2}$ to $u_{R}$.
Thus $q=f(q_{2},q_{2})=\dep$, where $f$ is the function
applied to $t$ by $\sigma_{M}$.

We have shown that $A$ accepts
$\enc(T_{M},\sigma_{M},\varphi_{M},\{Y_{i}\})$ if and 
only if $Y_{i}$ is independent in $M$, exactly as we wanted.

\textbf{Reducing to the connected case.}
Our final task is to show that we can construct a parse tree for $M$ when $M$ is
not connected.
In the first part of the proof, we have established that there is a
fixed-parameter tractable algorithm for constructing a parse tree relative to the
automaton $A$, when $M$ is connected.

We augment $A$ to obtain the automaton $A'$.
We add a new character, $\kappa$, to the alphabet of $A$,
and we add new states, $\dep'$ and $\indep'$, to its state space.
We augment the transition rules so that
$\delta_{2}(\kappa,\alpha,\beta)$ is
$\{\indep'\}$ when both $\alpha$ and $\beta$ are
$\indep'$ or accepting states of $A$, and set
$\delta_{2}(\kappa,\alpha,\beta)$ to be $\{\dep'\}$ otherwise.
The accepting states of $A'$ are the accepting states of $A$,
along with $\indep'$.
We can identify the connected components, $M_{1},\ldots, M_{n}$,
of $M$ in polynomial time \cite{BC95}.
We assume that $n>1$.
Each $M_{j}$ is in \mcal{M}, as \mcal{M} is minor-closed,
and we can construct a description $\Delta(M_{j})$ in polynomial
time, as $\Delta$ is minor-compatible.
Moreover, $\bw(M_{j})\leq \lambda$ \cite[Proposition 14.2.3]{Oxl11},
Therefore we have a fixed-parameter tractable algorithm for constructing
the parse trees, $(T_{M_{j}},\sigma_{M_{j}})$.
Now we construct a rooted tree with $n$ leaves, where each non-leaf has a left child and
a right child, and we apply the label $\kappa$ to each non-leaf vertex.
We then identify the $n$ leaves with the $n$ roots in
$T_{M_{1}},\ldots, T_{M_{n}}$.
Now it is straightforward to verify that $A'$ will use $A$ to check
independence in each connected component of $M$,
and accept if and only if $A$ accepts in each of those components.
Therefore $A'$ decides $\ind(X_{i})$ for any matroid in \mcal{M}.
Thus we have constructed a parse tree for $M$.
This completes the proof of \Cref{kibble}.
\end{proof}

\Cref{ermine} and \Cref{kibble} immediately lead to the following result.

\begin{theorem}
\label{marker}
Let \mcal{M} be a minor-closed class of matroids with a
minor-compatible representation, $\Delta$.
Assume that $\{M\in \mcal{M}\colon M\ \text{is $3$-connected}\}$
is efficiently pigeonhole.
Let $\psi$ be any sentence in \cmso.
There is a fixed-parameter tractable algorithm which will test
whether $\psi$ holds for matroids in \mcal{M}, where the
parameter is branch-width.
\end{theorem}

\section{Decidability and definability}
\label{decidability}

The theorems of Courcelle and Hlin\v{e}n\'{y} have as their goal
efficient model-checking: given a sentence and a
graph/matroid, we test whether the sentence is satisfied by that object.
Decidability is orthogonal to this problem:
given a class of objects and a sentence, we want to decide
(in finite time, but not necessarily efficiently) if that sentence is a theorem for the class.

\begin{definition}
\label{alaska}
Let \mcal{M} be a class of set-systems.
The \emph{\cmso\ theory} of \mcal{M} is the collection of
\cmso\ sentences that are satisfied by all set-systems in \mcal{M}.
We say that the \cmso\ theory of \mcal{M} is \emph{decidable} if there is
a Turing Machine which takes as input any sentence in
\cmso, and after a finite amount of time decides whether or not the sentence
is in the theory of \mcal{M}.
\end{definition}

The key idea in the forthcoming decidability proofs is that, given a tree automaton,
there is a finite procedure which will decide if there is a tree that the
automaton will accept.
(See, for example, \cite[Theorem 3.74]{Eng15}.)

\begin{lemma}
\label{fleece}
Let $A=(\Sigma, Q, F, \delta_{0},\delta_{2})$ be a tree automaton.
Let $Z$ be a subset of $Q$, and let $q$ be a state in $Q-Z$.
There is a finite procedure for deciding the following question:
does there exist a $\Sigma$\dash tree, $(T,\sigma)$, with root $t$
such that if $r$ is the run of $A$ on $(T,\sigma)$, then
$q\in r(t)$, and $r(v)\cap Z=\emptyset$ for every vertex $v$ of $T$.
\end{lemma}

\begin{proof}
Note that if $r(v)$ contains $q$, where $v$ is a non-root vertex, then we
may as well consider the subtree of $(T,\sigma)$ that has $v$ as its
root.
This means that we lose no generality in searching only for $\Sigma$\dash
trees such that $q$ is contained in $r(t)$, but not in $r(v)$ when $v$ is a
non-root vertex.
Our search will construct the desired tree, $T$, or establish that it does not exist.

We proceed by induction on $|Q-Z|$.
Assume $Q-Z$ contains only $q$.
If $\delta_{0}(\alpha)=\{q\}$ for some $\alpha\in \Sigma$,
then we return YES: we simply
consider the $\Sigma$\dash tree consisting of a single
leaf labelled with $\alpha$.
If no such $\alpha$ exists, then we return NO.
This completes the proof of the base case, so now we make the
obvious inductive assumption.

If $\delta_{0}(\alpha)\cap Z=\emptyset$ and $q\in\delta_{0}(\alpha)$
for some $\alpha\in \Sigma$, then we can construct the $\Sigma$\dash tree
with a single leaf labelled with $\alpha$, and the answer is YES.
Therefore we assume that no such $\alpha$ exists.

We search for tuples $(\alpha,q_{L},q_{R})\in\Sigma\times Q\times Q$
such that $q\in \delta_{2}(\alpha,q_{L},q_{R})$ and
$\delta_{2}(\alpha,q_{L},q_{R})\cap Z=\emptyset$.
If no such tuple exists, then we halt and return NO.
Otherwise, for each such tuple, we search for $\Sigma$\dash trees,
$(T_{L},\sigma_{L})$ and $(T_{R},\sigma_{R})$,
with the following properties:
If $r_{L}$ and $r_{R}$ are the runs on these trees, then
$r_{L}(v)\cap (Z\cup \{q\})=\emptyset$ for each vertex $v$ of $T_{L}$,
and similarly
$r_{R}(v)\cap (Z\cup \{q\})=\emptyset$.
Furthermore, $q_{L}$ is in $r_{L}(t_{L})$, and
$q_{R}$ is in $r_{R}(t_{R})$, where $t_{L}$ and $t_{R}$ are the roots of
$T_{L}$ and $T_{R}$.
By induction, we can construct such trees, if they exist.
If they do exist, then we construct $T$ from the disjoint union of $T_{L}$ and $T_{R}$
by adding a root $t$, and making its children $t_{L}$ and $t_{R}$.
We then apply the label $\alpha$ to $t$.
This justifies returning the answer YES.
If we find that no such trees exist
for each tuple $(\alpha,q_{L},q_{R})$,
then we return NO.
\end{proof}

\begin{corollary}
\label{alcove}
Let $A=(\Sigma, Q, F, \delta_{0},\delta_{2})$ be a tree automaton.
There is a finite procedure to decide whether there exists
a $\Sigma$\dash tree that $A$ accepts.
\end{corollary}

\begin{proof}
We repeatedly apply \Cref{fleece} with $Z$ set to be the empty set, and $q$
set to be a state in $F$.
\end{proof}

When we say that a class of set-systems is \emph{definable}
we mean there is a \cmso\ sentence, $\tau$, such that
a set-system satisfies $\tau$ if and only if it is in the class.
The matroid independence axioms can be stated in \cmso.
If $N$ is a fixed matroid, there is a \cmso\ sentence that characterises
the matroids having a minor isomorphic to $N$ \cite[Lemma 5.1]{Hli03b}.
From this it follows that a minor-closed class of matroids is definable if
it has a finite number of excluded minors.
There are also definable minor-closed classes that have infinitely
many excluded minors (\Cref{spread}).
The class of \mbb{K}\dash representable matroids is not
definable when \mbb{K} is an infinite field \cite{MNW18}.

\begin{theorem}
\label{carafe}
Let \mcal{M} be a definable class of set-systems with
bounded decomposition-width.
The \cmso\ theory of \mcal{M} is decidable.
\end{theorem}

\begin{proof}
Let $\psi$ be an arbitrary sentence in \cmso.
We wish to decide whether all set-systems in \mcal{M} satisfy $\psi$.
This is equivalent to deciding whether there exists a
set-system in \mcal{M} satisfying $\neg\psi$.
Let $\tau$ be a \cmso\ sentence such that
a set-system belongs to \mcal{M} if and only if it satisfies $\tau$.

\Cref{shader} implies that there is an $\{i\}$\dash ary automaton $A'$ such that
for every $M=(E,\mcal{I})$ in \mcal{M}, there is a $\Sigma$\dash tree
$(T_{M},\sigma_{M})$ and a bijection $\varphi_{M}\colon E\to L(T_{M})$ where
$A'$ accepts $\enc(T_{M},\sigma_{M},\varphi_{M},\{Y_{i}\})$ if and only if $Y_{i}$
is in \mcal{I}, for any $Y_{i}\subseteq E(M)$.
We use (the proof of) \Cref{family} to construct an automaton, $A$,
such that $A$ accepts $\enc(T,\sigma,\varphi,\emptyset)$
if and only if $M(A',T,\sigma,\varphi)$ satisfies $\tau\land \neg \psi$,
for each  $\Sigma$\dash tree $(T,\sigma)$ and each bijection
$\varphi$ from a finite set to the leaves of $T$.

According to \Cref{alcove}, we can decide in finite time whether or not there
is a $\Sigma$\dash tree that is accepted by $A$.
If $(T,\sigma)$ is such a tree, then let $\varphi$ be the identity function
on $L(T)$.
The set-system $M(A',T,\sigma,\varphi)$ satisfies
$\tau\land \neg \psi$, so it is a set-system in \mcal{M}
(as it satisfies $\tau$) that does not satisfy $\psi$.
Therefore $\psi$ is not in the theory of \mcal{M}.
On the other hand, if $M\in\mcal{M}$ does not satisfy $\psi$, then
$M=M(A',T_{M},\sigma_{M},\varphi_{M})$ satisfies
$\tau\land \neg \psi$, so $A$ will accept at least one tree,
namely $\enc(T_{M},\sigma_{M},\varphi_{M},\emptyset)$.
\end{proof}

\begin{corollary}
\label{quaker}
Let \mcal{M} be a definable pigeonhole class of matroids.
Let $\lambda$ be a positive integer.
The \cmso\ theory of $\{M\in \mcal{M}\colon \bw(M)\leq \lambda\}$
is decidable.
\end{corollary}

\begin{proof}
The family of matroids with branch-width at most $\lambda$
is minor-closed, and it has finitely many excluded
minors \cite{GGRW03}.
Therefore we can let $\tau$ be a \cmso\ sentence
encoding the independence axioms for matroids,
membership of \mcal{M},
and branch-width of at most $\lambda$.
Thus a set-system satisfies $\tau$ if and only if it is
in $\{M\in \mcal{M}\colon \bw(M)\leq \lambda\}$.
This class has bounded decomposition-width, so we apply \Cref{carafe}.
\end{proof}

The resolution of Rota's conjecture \cite{GGW14} means that the class of
\mbb{F}\dash representable matroids is definable
when \mbb{F} is a finite field.
\Cref{quaker} now implies that the \cmso\ theory of \mbb{F}\dash representable
matroids with branch-width at most $\lambda$ is decidable.
This was previously proved by Hlin\v{e}n\'{y} and Seese
\cite[Corollary 5.3]{HS06}, who did not rely on Rota's conjecture.
Let \mcal{M} be any minor-closed class of \mbb{F}\dash representable
matroids, where \mbb{F} is a finite field.
Geelen, Gerards, and Whittle have also announced that
\mcal{M} has finitely many excluded
minors \cite[Theorem 6]{GGW14}.
Therefore \mcal{M} is definable, and hence
the \cmso\ theory of $\{M\in\mcal{M}\colon \bw(M)\leq \lambda\}$
is decidable, for any positive integer $\lambda$.

Let $Z$ be a flat of the matroid $M$.
If the restriction of $M$ to $Z$ contains no coloops, then $Z$ is a \emph{cyclic flat}.
A basis, $B$, of $M$ is \emph{fundamental} if
$B\cap Z$ spans $Z$ whenever $Z$ is a cyclic flat.
A matroid with a fundamental basis is a
\emph{fundamental transversal matroid} (see \cite{BKM11}).
It is an exercise to prove that $B$ is fundamental if and only if
$x$ is freely placed in the flat spanned by the fundamental circuit, $C(x,B)$,
for every $x\notin B$.
This is equivalent to saying that if $Z$ is a cyclic flat containing $x$,
then $Z$ contains $C(x,B)$.
This property can clearly be expressed in \mso.
In \cite[Theorem 6.3]{FMN-II}, we prove that the class of fundamental transversal
matroids is pigeonhole.
Therefore the \cmso\ theory of fundamental transversal matroids with
branch-width at most $\lambda$ is decidable by \Cref{quaker}.

The class of bicircular matroids can be characterised by an
\mso\ sentence \cite{FMN-III}.
The class is also pigeonhole, as we prove in
\cite[Theorem 8.4]{FMN-II}, so the \cmso\ theory of bicircular matroids with
branch-width at most $\lambda$ is decidable.

\subsection{Undecidable theories}
We also have some results that allow us to prove results in the
negative direction, by showing that certain classes have undecidable theories.

Assume that $F$ is a flat of the matroid $M$, and let
$M'$ be a single-element extension of $M$.
Let $e$ be the element in $E(M')-E(M)$.
We say that $M'$ is a \emph{principal extension} of $M$
by $F$ if $F\cup e$ is a flat of $M'$ and whenever
$X\subseteq E(M)$ spans $e$ in $M'$, it spans $F\cup e$.

Let $G$ be a simple graph with vertex set $\{v_{1},\ldots, v_{n}\}$
and edge set $\{e_{1},\ldots, e_{m}\}$.
Let $m(G)$ be the rank\dash $3$ sparse paving matroid with ground set
$\{v_{1},\ldots, v_{n}\}\cup\{e_{1},\ldots, e_{m}\}$.
The only non-spanning circuits of $m(G)$ are the
sets $\{v_{i},e_{k},v_{j}\}$, where $e_{k}$ is an
edge of $G$ joining the vertices $v_{i}$ and $v_{j}$.

\begin{theorem}
\label{sizzle}
Let \mcal{M} be a class of matroids
that contains all rank\dash $3$ uniform matroids,
and is closed under principal extensions.
The \mso\ theory of \mcal{M} is undecidable.
\end{theorem}

\begin{proof}
We let $G$ be a simple graph, and we let $m^{+}(G)$ be the
matroid obtained from $m(G)$ by placing a new element parallel to each
`vertex' element $v_{i}$.
It is easy to check that $m^{+}(G)$ is contained in \mcal{M}
for every simple graph $G$.
Moreover, we can characterise the matroids that are equal to
$m^{+}(G)$ for some graph $G$ in the following way.
Let $M$ be a matroid.
Then $M$ is equal to $m^{+}(G)$ for some graph $G$ (with at least three vertices)
if and only if the following properties hold:
\begin{enumerate}[label=\textup{(\roman*)}]
\item $r(M)=3$,
\item $M$ is loopless, and any rank\dash $1$ flat has cardinality one or two,
\item any rank\dash $2$ flat that contains at least three rank\dash $1$ flats
contains exactly three such flats, one of cardinality one, and two of cardinality two,
\item if $x$ is an element that is not in a parallel pair, then $x$ is in exactly one
rank\dash $2$ flat that contains three rank\dash $1$ flats.
\end{enumerate}
From this it is clear that there is an \mso\ sentence, $\tau$, such that
a matroid satisfies $\tau$ if and only if it is isomorphic to
$m^{+}(G)$ for some simple graph $G$.

Now we consider the logical language, \msone, for graphs.
In this language, we can quantify over variables that represent
vertices, and variables that represent subsets of vertices.
We have a binary predicate that expresses when a vertex is in
a set of vertices, and another that expresses when two vertices are
adjacent.
Let $\psi$ be a sentence in \msone.
There is a corresponding sentence, $\psi'$, in \mso\ such that
a simple graph, $G$, satisfies $\psi$ if and only if $m^{+}(G)$
satisfies $\psi'$.
Let $\vertex(X)$ stand for an \mso\ formula expressing the fact that
$X$ is a $2$\dash element circuit.
To construct $\psi'$, we make the following interpretations:
\begin{enumerate}[label=\textup{(\roman*)}]
\item when $v$ is a vertex variable, we replace $\exists v$ with $\exists X_{v} \vertex(X_{v}) \land $, and replace $\forall v$ with
$\forall X_{v} \vertex(X_{v})\to$.
\item when $V$ is a set variable, we replace $\exists V$ with
\begin{multline*}
\rule{32pt}{0pt}\exists X (\forall X_{1} (\sing(X_{1}) \land X_{1}\subseteq X)\to\\
\exists X_{2} (X_{1}\subseteq X_{2}\land X_{2}\subseteq X\land \vertex(X_{2})))\ \land
\end{multline*}
and we replace $\forall V$ with
\begin{multline*}
\rule{32pt}{0pt}\forall X (\forall X_{1} (\sing(X_{1}) \land X_{1}\subseteq X)\to\\
\exists X_{2} (X_{1}\subseteq X_{2}\land X_{2}\subseteq X\land \vertex(X_{2})))\ \to
\end{multline*}
\item we replace $v\sim v'$ (the adjacency predicate) with
an \mso\ formula saying
that $X_{v}\cup X_{v'}$ is not a flat.
\end{enumerate}

Trahtenbrot's Theorem shows that there is no finite procedure for
deciding whether there is a finite structure that satisfies a given
first-order sentence.
A corollary is that the first-order theory of finite simple graphs
is undecidable (see \cite[Theorem 6.2.2]{Lib04}).
Since any first-order sentence is also a sentence in \msone, it follows that
the \msone\ theory of finite simple graphs is undecidable.
Imagine that the \mso\ theory of \mcal{M} is decidable.
Then for any \msone\ sentence, $\psi$, we could decide
if $\tau\to \psi'$ is a theorem for \mcal{M}.
This is equivalent to deciding whether $\psi$ is a theorem for
all simple graphs, so we have contradicted Trahtenbrot's result.
\end{proof}

The argument in \cite[Proposition 6.1]{FMN-II} shows that the class of
(strict) gammoids is closed under principal extensions.
The next result follows easily.

\begin{corollary}
\label{pulpit}
Let \mbb{K} be an infinite field.
The \mso\ theory of rank\dash $3$ \mbb{K}\dash representable matroids
is undecidable.
The \mso\ theories of rank\dash $3$ cotransversal matroids
and gammoids are undecidable.
Therefore the \mso\ theory of corank\dash $3$ transversal
matroids is undecidable.
\end{corollary}

\Cref{pulpit} is complementary to a result by Hlin\v{e}n\'{y} and Seese
\cite[Theorem 7.2]{HS06}, who have shown that the \mso\ theory of spikes is
undecidable.
Although every spike has branch-width three, spikes
have unbounded rank, and are not representable over a
common field.

\begin{remark}
If $L\subseteq L'$ are languages, and a class of objects has an undecidable
$L$ theory, then it obviously follows that it has an undecidable $L'$ theory.
So any class of matroids with an undecidable \mso\ theory has an undecidable
\cmso\ theory.
\end{remark}

\subsection{Definability and pigeonhole classes}
We will now show that definability and the pigeonhole property
are independent of each other by exhibiting classes that have exactly one of these
properties.
One case is easy: the class of sparse paving matroids can 
be defined by insisting that every non-spanning circuit is
a hyperplane.
But \cite[Lemma 4.1]{FMN-II} implies this class is not pigeonhole.

For the other case, we consider \emph{polygon} matroids.
Let $C_{n}$ be the rank\dash $3$ sparse paving matroid with
ground set $\{e_{1},\ldots ,e_{2n}\}$ and
non-spanning circuits
\[
\{e_{1},e_{2},e_{3}\},\{e_{3},e_{4},e_{5}\},\ldots,
\{e_{2n-3},e_{2n-2},e_{2n-1}\},\{e_{2n-1},e_{2n},e_{1}\}.
\]
The polygon matroids form a well-known infinite antichain
\cite[Example 14.1.2]{Oxl11}.
Next we consider \emph{path} matroids.
They too are rank\dash $3$ sparse paving matroids.
The ground set of a path matroid can be ordered
$e_{1},\ldots, e_{n}$ in such a way that
any non-spanning circuit is a set of three consecutive
elements in the ordering.
Note that the sparse paving property implies that if
$\{e_{i},e_{i+1},e_{i+2}\}$ is a non-spanning circuit, then
$\{e_{i+1},e_{i+2}, e_{i+3}\}$ is not.
Every rank\dash $3$ proper minor of a polygon matroid is a path matroid,
and every path matroid is a minor of a polygon matroid.

\begin{remark}
\label{spread}
Path matroids can be characterised as follows:
they are rank-three matroids where every non-spanning circuit
contains exactly three elements, and no element is in more than
two non-spanning circuits.
Moreover, any non-spanning circuit contains an element that is in
exactly one non-spanning circuit.
Finally, if $\{u,v,w\}$ is a non-spanning circuit, and $w$ is in exactly one
non-spanning circuit, then there is a partition, $(U,V)$, of $E(M)-w$
with $u\in U$ and $v\in V$, where $\{u,v,w\}$ is the only non-spanning circuit
containing elements of both $U$ and $V$.
This characterisation shows that the class of path matroids is \mso\dash definable.
Furthermore, matroids with rank at most two can be characterised by saying that any
subset containing three pairwise disjoint singleton sets is dependent.
Therefore the minor-closed class consisting of path matroids and all
matroids with rank at most two is definable.
It has an infinite number of excluded minors, since every
polygon matroid is an excluded minor.
\end{remark}

\begin{proposition}
\label{dubber}
Let \mcal{M} be the class containing all polygon and path matroids, and
all matroids of rank at most two.
Then \mcal{M} has bounded decomposition-width.
\end{proposition}

\begin{proof}
It is easy to see that if $M$ is a matroid with rank zero or one, then
$\dw(M)\leq 3$.
Now assume that $M$ is a rank\dash $2$ matroid.
Let $(T,\varphi)$ be a decomposition of $M$ such that if
$(U,V)$ is a displayed partition, then
no more than one parallel class of $M$ intersects both $U$ and $V$.
It is clear that such a decomposition exists.
Now we can easily verify that $\dw(M)\leq 5$.

Let $M$ be a rank\dash $3$ matroid in \mcal{M}.
It is easy to see that there is an ordering $e_{1},\ldots, e_{n}$
of $E(M)$ such that for any partition
$(U,V)=(\{e_{1},\ldots, e_{t}\},\{e_{t+1},\ldots, e_{n}\})$,
at most two non-spanning circuits contain elements
from both $U$ and $V$.
We let $(T,\varphi)$ be a decomposition that displays only
partitions of this type.
It is straightforward to verify that this type of decomposition
leads to an upper bound on the decomposition-width of all
rank\dash $3$ matroids in \mcal{M}.
\end{proof}

Now we can prove the existence of a minor-closed class that is
pigeonhole without being definable.

\begin{lemma}
\label{golfer}
There is a minor-closed class, \mcal{M}, of matroids with the
following properties.
Each matroid in \mcal{M} has rank (and therefore branch-width)
at most three.
Furthermore \mcal{M} has bounded decomposition-width, and is
therefore pigeonhole.
However, \mcal{M} has an undecidable \mso\ theory, so
\mcal{M} is not definable.
\end{lemma}

\begin{proof}
We consider the classes consisting of all matroids with rank at most two,
all path matroids, and some subset of the polygon matroids.
The fact that these classes have bounded decomposition-width
follows from \Cref{dubber}.
The number of such classes is the cardinality of the power set of the
natural numbers, so there are uncountably many such classes.
Assume that every such class has a decidable theory.
There are countably many Turing Machines, so
we let $\mcal{M}_{1}$ and $\mcal{M}_{2}$ be two distinct such classes,
such that exactly the same Turing Machine decides the theories of
$\mcal{M}_{1}$ and $\mcal{M}_{2}$.
Therefore $\mcal{M}_{1}$ and $\mcal{M}_{2}$ satisfy exactly the same
\mso\ sentences.
Without loss of generality, we can let $P$ be a polygon matroid in
$\mcal{M}_{1}$ but not in $\mcal{M}_{2}$.
It is easy to express the statement
`this matroid is not isomorphic to $P$' in \mso.
This sentence is a theorem for $\mcal{M}_{2}$, but not for
$\mcal{M}_{1}$, so we have
a contradiction.
Therefore there are classes that have undecidable theories.
Any class fails to be definable by virtue of \Cref{quaker}.
\end{proof}

\subsection{Open problems}

We start this section of open problems by recollecting \Cref{spring}.

\begin{conjecture}
Let \mcal{M} be a strongly pigeonhole class of matroids.
Then $\{M^{*}\colon M\in\mcal{M}\}$ is strongly pigeonhole.
\end{conjecture}

The class of lattice path matroids has infinitely many
excluded minors \cite{Bon10}.
Despite this, we make the following conjecture.

\begin{conjecture}
\label{shnook}
The class of lattice path matroids can be characterised by
a sentence in \mso.
\end{conjecture}

We have proved that the class of lattice path matroids is pigeonhole
\cite[Theorem 7.2]{FMN-II}, so \Cref{shnook}, along with \Cref{quaker}, would
imply the decidability of the \mso\ theory of lattice path matroids
with bounded branch-width.

The following conjecture would imply that any minor-closed class of
\hgg\ matroids can be characterised with a sentence in \mso, when $H$
is a finite group.

\begin{conjecture}
\label{roadie}
Let $H$ be a finite group, and let \mcal{M} be a minor-closed class
of \hgg\ matroids.
Then \mcal{M} has only finitely many excluded minors.
\end{conjecture}

We also conjecture that the class of \hgg\ matroids is (efficiently)
pigeonhole (\cite[Conjecture 9.3]{FMN-II}).
This conjecture, combined with \Cref{roadie}, would imply
decidability for any minor-closed of \hgg\ matroids
with bounded branch-width.

DeVos, Funk, and Pivotto have proved that
if $H$ is an infinite group, then the class of \hgg\ matroids has infinitely
many excluded minors \cite[Corollary 1.3]{dVFP14}.
We conjecture a stronger property.

\begin{conjecture}
\label{gnomon}
Let $H$ be an infinite group.
The class of \hgg\ matroids cannot be
characterised with a sentence in \cmso.
\end{conjecture}

It is not too difficult to see that the techniques of \cite{MNW18} settle
this conjecture when $H$ contains elements of arbitrarily high order.
Thus it is open only in the case that $H$ is an infinite group
with finite exponent.
An easy example of such a group is the infinite direct product
$(\mbb{Z}/2\mbb{Z})^{\mbb{Z}^{+}}$, but there exist more sophisticated examples,
such as Tarski monster groups.

We also believe the following.

\begin{conjecture}
\label{mucker}
Let $H$ be an infinite group.
The class of rank\dash $3$ \hgg\ matroids has an undecidable
\mso\ theory.
\end{conjecture}

\section{Acknowledgements}

We thank Geoff Whittle for several important conversations, and the referees for their careful reading and helpful comments.
Funk and Mayhew were supported by a Rutherford Discovery Fellowship, managed by Royal Society Te Ap\={a}rangi.



\end{document}